\documentclass[preprint]{elsarticle}

\usepackage[T1]{fontenc}
\usepackage[utf8]{inputenc}
\usepackage{amsmath,amsthm,amssymb}
\usepackage{amsfonts,mathrsfs}
\usepackage{todonotes}
\usepackage{lineno}
\usepackage{xcolor}
\usepackage{tikz}
\usetikzlibrary{decorations.pathreplacing,calc,patterns}
\usepackage{subcaption}
\usepackage{enumerate}
\usepackage{stmaryrd}
\usepackage{fontawesome}
\usepackage[ruled,vlined,linesnumbered]{algorithm2e}
\usepackage{hyperref}
\usepackage[capitalise,nameinlink, noabbrev]{cleveref}
\usepackage{caption}
\colorlet{darkGreen}{green!50!black}

\crefalias{AlgoLine}{line}%

\makeatletter
\let\cref@old@stepcounter\stepcounter
\def\stepcounter#1{%
  \cref@old@stepcounter{#1}%
  \cref@constructprefix{#1}{\cref@result}%
  \@ifundefined{cref@#1@alias}%
    {\def\@tempa{#1}}%
    {\def\@tempa{\csname cref@#1@alias\endcsname}}%
  \protected@edef\cref@currentlabel{%
    [\@tempa][\arabic{#1}][\cref@result]%
    \csname p@#1\endcsname\csname the#1\endcsname}}
\makeatother

\usepackage{ifthen}
\usepackage{skak}
\usepackage{cancel}
\usepackage[inline]{enumitem}
\usepackage{MnSymbol} %

\colorlet{myGreen}{green!50!black}
\definecolor{myBlue}{rgb}{0.25, 0.0, 1.0}
\definecolor{lgray}{rgb}{0.75, 0.75, 0.75}
\newcommand{\ourtitle}{Impartial Chess}

\hypersetup{
  colorlinks=true,
  linkcolor=myBlue,
  citecolor=myGreen,
  urlcolor=myGreen,
  bookmarksopen=true,
  bookmarksnumbered,
  bookmarksopenlevel=2,
  bookmarksdepth=3
  pdftitle = {\ourtitle},
  pdfauthor= {Eric Gottlieb, Matjaz Krnc, Peter Mursic},
}

\usepackage{thmtools}
\usepackage{thm-restate}

\newtheorem{theorem}{Theorem}[section]
\newtheorem{definition}[theorem]{Definition}
\newtheorem{lemma}[theorem]{Lemma}
\newtheorem{proposition}[theorem]{Proposition}
\newtheorem{corollary}[theorem]{Corollary}
\newtheorem{conjecture}[theorem]{Conjecture}
\newtheorem{example}[theorem]{Example}

\newtheorem{observation}[theorem]{Observation}

\newcommand{\R}{\symrook}
\newcommand{\Pawn}{\sympawn}
\newcommand{\Bishop}{\symbishop}
\newcommand{\Knight}{\symknight}

\newcommand{\C}{\mathscr G^{\pm}}
\newcommand{\K}{\symking}

\newcommand{\Q}{\symqueen}
\newcommand{\Queen}{\symqueen}
\DeclareMathOperator{\Mex}{mex}
\newcommand{\D}{d}

\newcommand{\trunc}{\textnormal{trunc}}
\newcommand{\br}[1]{\llbracket #1 \rrbracket}

\newcommand{\DAG}[2]{\textnormal{DAG}_{#1}(#2)}

\newcommand{\lp}{\mathrm{lp}}

\def\P{\mathcal{P}}
\def\N{\mathcal{N}}
\def\G{\mathbb{SG}}
\def\Dr{\mathcal{D}}
\def\SG{\G}

\def\Young{Young diagram}
\def\Youngs{Young diagrams}
\def\Y{\mathbb{Y}}

    \newcommand{\floor}[1]{\left\lfloor #1 \right\rfloor}

\newcommand{\gama}{%
\begin{tikzpicture}[scale=0.08]%
\draw (0,0) -- (0,-3)--(1,-3)--(1,0)
    (0,0)--(3,0)--(3,-1)--(0,-1)
    (0,-2)--(1,-2)
    (2,0)--(2,-1);%
\end{tikzpicture}%
}

\newcommand{\genstaircase}{%
\begin{tikzpicture}[scale=0.06]%
\draw (0,0) -- (0,-4)--(2,-4)--(2,-2)--(4,-2)--(4,0)--(0,0)
    (1,0)--(1,-4)
    (0,-1)--(4,-1)
    (2,0)--(2,-2)--(0,-2)
    (3,0)--(3,-2)    
    (0,-3)--(2,-3);
\draw[white] (4.4,0)--(4.4,-1);
\end{tikzpicture}%
}

\newcommand{\staircase}{%
\begin{tikzpicture}[scale=0.08]%
\draw (0,0) -- (0,-3)--(1,-3)--(1,0)
    (0,0)--(3,0)--(3,-1)--(0,-1)
    (0,-2)--(2,-2)
    (2,0)--(2,-2);%
\draw[white] (3.4,0)--(3.4,-1);
\end{tikzpicture}%
}

\newcommand{\rectangle}{%
\begin{tikzpicture}[scale=0.08]%
\draw (0,0) -- (0,-3)--(1,-3)--(1,0)
    (0,0)--(3,0)--(3,-1)--(0,-1)
    (0,-2)--(3,-2)--(3,-3)--(0,-3)
    (2,0)--(2,-3)--(3,-3)--(3,0);%
\draw[white] (3.4,0)--(3.4,-1);
\end{tikzpicture}%
}

\newcommand{\fgenstaircase}[3]{\genstaircase^{#3}_{#1,#2}}
\newcommand{\fstaircase}[1]{\staircase^{#1}}
\newcommand{\frectangle}[2]{\rectangle_{#1,#2}}
\newcommand{\fgama}[2]{\gama_{#1,#2}}
\newcommand{\gs}[3]{
\fgenstaircase{#1}{#2}{#3}
}

\newcommand{\picturefamilies}{
    \begin{tikzpicture}[]
 \begin{scope}[shift={(4,0)},scale=0.5]
\def\x{6,6,6,6}
\foreach \i [count=\xi] in \x{
    \draw [step=1.0, very thick] (\xi-1,0) grid (\xi,-\i);
}
 \end{scope}
 
  \begin{scope}[shift={(7.5,0)},scale=0.5]
\def\x{6,1,1,1}
\foreach \i [count=\xi] in \x{
    \draw [step=1.0, very thick] (\xi-1,0) grid (\xi,-\i);
}
 \end{scope}
 
  \begin{scope}[shift={(0,0)},scale=0.5]
\def\x{6,5,4,3,2,1}
\foreach \i [count=\xi] in \x{
    \draw [step=1.0, very thick] (\xi-1,0) grid (\xi,-\i);
}
 \end{scope}

   \begin{scope}[shift={(10.5,0)},scale=0.5]
\def\x{6,6,3,3}
\foreach \i [count=\xi] in \x{
    \draw [step=1.0, very thick] (\xi-1,0) grid (\xi,-\i);
}
 \end{scope}

\end{tikzpicture}
}

\newcommand{\picericferrer}{
\begin{tikzpicture}[scale=0.5, every node/.style={scale=1}]
 \begin{scope}[shift={(0,0)},scale=1]
\def\x{6,4,3,3,1}
\foreach \i [count=\xi] in \x{
    \draw [step=1.0, very thick] (\xi-1,0) grid (\xi,-\i);
}
\end{scope}
 \begin{scope}[shift={(7,0)},scale=1]
\def\x{5,4,4,2,1,1}
\foreach \i [count=\xi] in \x{
    \draw [step=1.0, very thick] (\xi-1,0) grid (\xi,-\i);
}
\end{scope}
 \begin{scope}[shift={(15,0)},scale=1]
\def\x{3,3,1}
\foreach \i [count=\xi] in \x{
    \draw [step=1.0, very thick] (\xi-1,0) grid (\xi,-\i);
}
\end{scope}
\end{tikzpicture}
}

\newcommand{\class}{

\begin{tikzpicture}[scale=0.5]
\def\unknown{unknown}
\def\textsizes{0.6}
\draw[black] (1,0)--node[fill=none,yshift=1mm,xshift=12mm,scale=\textsizes]{domestic}++(24,0)--++(0,13)--++(-24,0)--cycle;
\draw [blue] (1.5,2)--node[fill=none,yshift=1mm,xshift=12mm,scale=\textsizes]{tame}++(23,0)--++(0,10.5)--++(-23,0)--cycle;
\draw [darkGreen] (2,4)--node[fill=none,yshift=1mm,xshift=12mm,scale=\textsizes]{miserable}++(22,0)--++(0,8)--++(-22,0)--cycle;
\draw [red] (3,7)--node[fill=none,yshift=1mm,xshift=22mm,scale=\textsizes]{pet}++(16,0)--++(0,4)--++(-16,0)--cycle;
\draw [gray] (3.5,-2)--node[fill=none,yshift=1mm,scale=\textsizes]{forced}++(8,0)--++(0,12.5)--++(-8,0)--cycle;
\draw [cyan] (2.5,-2.5)--node[fill=none,yshift=1mm,xshift=-17.5mm,,scale=\textsizes]{returnable}++(17,0)--++(0,14)--++(-17,0)--cycle;

\node[scale=1] at (7.5,9.5) {$\Bishop$, $(\Pawn,\rectangle)$, $(\R,\staircase)$,};
\node[scale=1] at (7.5,8.5) {$(\Queen,\staircase)$, $(\Knight,\staircase)$,};
\node[scale=1] at (7.5,7.5) {$(\Dr,\staircase)$, $(\Dr,\rectangle)$};

\node at (7.5,6) {\textsc{Nim}, $(\K,\rectangle)$,};
\node at (7.5,5) {$(\R,\rectangle)$, $(\R,\genstaircase)$};
\node at (7.5,3) {$(\R,S_1)$};
\node at (7.5,1) {\unknown};
\node at (7.5,-1) {\unknown};

\node at (15.27,9.5) {\textsc{Sub},  $(\Dr,\genstaircase)$, $(\Knight,\rectangle)$,};
\node at (15.3,8.5) {$\Pawn$, $(\Pawn,\staircase)$,  $(\K,\staircase)$};

\node at (15.3,6) {$\textsc{Wyt}=(\Queen,\rectangle),$};
\node at (15.3,5) {$(\Queen,\genstaircase)$};
\node at (15.3,3) {$(\Queen,S_2)$};
\node at (15.45,1){$\Dr$, \textsc{LCTR}, $(\Knight,\genstaircase)$, $\Knight$};
\node at (15.3,-1) {\unknown};

\node at (21.8,7.8) {\unknown};
\node at (21.8,3) {\unknown};
\node at (21.8,1) {\unknown};
\node at (22,-0.8) {$\R$, \textsc{Mark},};
\node at (22,-1.8) {$\Queen$, $(\K,\genstaircase)$};
\end{tikzpicture}
}

\newcommand{\knightmoves}{
\begin{tikzpicture}[scale=0.3]
\begin{scope}
\def\x{14,14,13,12,12,12,12,9,5,3}
\foreach \i [count=\xi] in \x{
    \draw [step=1.0,black] (0,1-\xi) grid (\i,-\xi);}
\draw[fill=gray!40,pattern=north east lines] (0,-1)--++(1,0)--++(0,1)--++(-1,0)--cycle;
\draw[fill=gray!40,pattern=north east lines] (2,-2)--++(1,0)--++(0,1)--++(-1,0)--cycle;
\draw[fill=gray!40,pattern=north east lines] (4,-3)--++(1,0)--++(0,1)--++(-1,0)--cycle;
\draw[fill=gray!40,pattern=north east lines] (6,-4)--++(1,0)--++(0,1)--++(-1,0)--cycle;
\draw[fill=gray!40,pattern=north east lines] (8,-5)--++(1,0)--++(0,1)--++(-1,0)--cycle;
\draw[fill=gray!40,pattern=north east lines] (10,-6)--++(1,0)--++(0,1)--++(-1,0)--cycle;

\draw[fill=gray!40,pattern=north east lines] (1,-3)--++(1,0)--++(0,1)--++(-1,0)--cycle;
\draw[fill=gray!40,pattern=north east lines] (3,-4)--++(1,0)--++(0,1)--++(-1,0)--cycle;
\draw[fill=gray!40,pattern=north east lines] (5,-5)--++(1,0)--++(0,1)--++(-1,0)--cycle;
\draw[fill=gray!40,pattern=north east lines] (7,-6)--++(1,0)--++(0,1)--++(-1,0)--cycle;
\draw[fill=gray!40,pattern=north east lines] (9,-7)--++(1,0)--++(0,1)--++(-1,0)--cycle;

\draw[fill=gray!40,pattern=north east lines] (2,-5)--++(1,0)--++(0,1)--++(-1,0)--cycle;
\draw[fill=gray!40,pattern=north east lines] (4,-6)--++(1,0)--++(0,1)--++(-1,0)--cycle;
\draw[fill=gray!40,pattern=north east lines] (6,-7)--++(1,0)--++(0,1)--++(-1,0)--cycle;
\draw[fill=gray!40,pattern=north east lines] (8,-8)--++(1,0)--++(0,1)--++(-1,0)--cycle;

\draw[fill=gray!40,pattern=north east lines] (3,-7)--++(1,0)--++(0,1)--++(-1,0)--cycle;
\draw[fill=gray!40,pattern=north east lines] (5,-8)--++(1,0)--++(0,1)--++(-1,0)--cycle;
\draw[fill=gray!40,pattern=north east lines] (4,-9)--++(1,0)--++(0,1)--++(-1,0)--cycle;
\draw[fill=gray!40,pattern=north west lines](0,0)--(0,-1)--(1,-1)--(1,-2)--(2,-2)--(2,-3)--(3,-3)--(3,-4)--(4,-4)--(4,-5)--(5,-5)--(5,-6)--(6,-6)--(6,-7)--(7,-7)--(7,-8)--(9,-8)--(9,-7)--(12,-7)--(12,-3)--(13,-3)--(13,-2)--(14,-2)--(14,0)--cycle;

\end{scope}
\begin{scope}[shift={(17,0)}]
\def\x{6,5,4,2,1}
\foreach \i [count=\xi] in \x{
    \draw [step=1.0,black] (0,1-\xi) grid (\i,-\xi);}
\draw[fill=gray!40,pattern=north east lines](0,0)--(6,0)--(6,-1)--(5,-1)--(5,-2)--(4,-2)--(4,-3)--(2,-3)--(2,-4)--(1,-4)--(1,-5)--(0,-5)--cycle;;
\end{scope}

\begin{scope}[shift={(26,0)}]
\def\x{14,13,11,9,8,7,6,2}
\foreach \i [count=\xi] in \x{
    \draw [step=1.0,black] (0,1-\xi) grid (\i,-\xi);}
\draw[fill=gray!40,pattern=north west lines](0,0)--(14,0)--(14,-1)--(13,-1)--(13,-2)--(11,-2)--(11,-3)--(9,-3)--(9,-4)--(8,-4)--(8,-5)--(7,-5)--(7,-6)--(6,-6)--(6,-7)--(2,-7)--(2,-8)--(0,-8)--cycle;
\end{scope}
\end{tikzpicture}
}

\newcommand{\knightsquare}{
\begin{tikzpicture}[scale=0.2]
\draw[fill=gray!40,pattern=north east lines] (0,-1)--++(1,0)--++(0,1)--++(-1,0)--cycle;
\end{tikzpicture}
}

\newcommand{\pawnsquare}{
\begin{tikzpicture}[scale=0.2]
\draw[fill=gray!40,pattern=north west lines] (0,-1)--++(1,0)--++(0,1)--++(-1,0)--cycle;
\end{tikzpicture}
}

\newcommand{\piclabelequiv}{
\begin{tikzpicture}[scale=0.7, every node/.style={scale=0.7}]
\begin{scope}
\def\x{3,3,2}
\foreach \i [count=\xi] in \x{
    \draw [step=1.0,black] (0,1-\xi) grid (\i,-\xi);}
\node (3) at (0+0.5,0-0.5){$(0,0)$};
\node (2) at (2+0.5,-1-0.5){$(1,2)$};
\node (1) at (1+0.5,-2-0.5){$(2,1)$};
\draw[->] (3)--node[fill=none,yshift=-1mm,scale=0.7] {(1,2)}(2);
\draw[->] (3)--node[fill=none,yshift=-1mm,scale=0.7] {(2,1)}(1);
\end{scope}

\begin{scope}[shift={(4,0)}]
\def\x{4,3,3}
\foreach \i [count=\xi] in \x{
    \draw [step=1.0,black] (0,1-\xi) grid (\i,-\xi);}
\node (3) at (0+0.5,0-0.5){$(0,0)$};
\node (2) at (2+0.5,-1-0.5){$(1,2)$};
\node (1) at (1+0.5,-2-0.5){$(2,1)$};
\draw[->] (3)--node[fill=none,yshift=-1mm,scale=0.7] {(1,2)}(2);
\draw[->] (3)--node[fill=none,yshift=-1mm,scale=0.7] {(2,1)}(1);
\end{scope}
\end{tikzpicture}
}

\newcommand{\piclabelnotequiv}{

\begin{tikzpicture}[scale=0.7, every node/.style={scale=0.7}]
\begin{scope}
\def\x{2,2,2}
\foreach \i [count=\xi] in \x{
    \draw [step=1.0,black] (0,1-\xi) grid (\i,-\xi);}
\node (3) at (0+0.5,0-0.5){$(0,0)$};
\node (1) at (1+0.5,-2-0.5){$(2,1)$};
\draw[->] (3)--node[fill=none,yshift=-1mm,scale=0.7] {(2,1)}(1);
\end{scope}

\begin{scope}[shift={(4,0)}]
\def\x{3,3}
\foreach \i [count=\xi] in \x{
    \draw [step=1.0,black] (0,1-\xi) grid (\i,-\xi);}
\node (3) at (0+0.5,0-0.5){$(0,0)$};
\node (2) at (2+0.5,-1-0.5){$(1,2)$};
\draw[->] (3)--node[fill=none,yshift=-1mm,scale=0.7] {(1,2)}(2);
\end{scope}
\end{tikzpicture}
}

\newcommand{\picslambdak}{
\begin{tikzpicture}[scale=0.5]
\def\x{9,7,6,4,3,1}
\foreach \i [count=\xi] in \x{
    \draw [step=1.0,black] (0,1-\xi) grid (\i,-\xi);}
\def\x{3,1}
\begin{scope}[shift={(6,0)}]
\foreach \i [count=\xi] in \x{
    \draw [step=1.0,black,ultra thick] (0,1-\xi) grid (\i,-\xi);}
\end{scope}
\begin{scope}[shift={(3,-2)}]
\foreach \i [count=\xi] in \x{
    \draw [step=1.0,black,ultra thick] (0,1-\xi) grid (\i,-\xi);}
\end{scope}
\begin{scope}[shift={(0,-4)}]
\foreach \i [count=\xi] in \x{
    \draw [step=1.0,black,ultra thick] (0,1-\xi) grid (\i,-\xi);}
\end{scope}
\end{tikzpicture}
}

\newcommand{\picsgyoung}{
\begin{tikzpicture}[scale=0.5]
\begin{scope}
\def\x{5,5,3,2,1}
\foreach \i [count=\xi] in \x{
    \draw [step=1.0,black] (0,1-\xi) grid (\i,-\xi);}
\node at (0.5,-0.5){$2$};
\node at (1.5,-0.5){$1$};
\node at (2.5,-0.5){$0$};
\node at (3.5,-0.5){$2$};
\node at (4.5,-0.5){$1$};

\node at (0.5,-1.5){$0$};
\node at (1.5,-1.5){$3$};
\node at (2.5,-1.5){$2$};
\node at (3.5,-1.5){$1$};
\node at (4.5,-1.5){$0$};

\node at (0.5,-2.5){$2$};
\node at (1.5,-2.5){$1$};
\node at (2.5,-2.5){$0$};

\node at (0.5,-3.5){$1$};
\node at (1.5,-3.5){$0$};

\node at (0.5,-4.5){$0$};

\end{scope}

\begin{scope}[shift={(10,0)}]
\def\x{5,5,3,2,1}
\foreach \i [count=\xi] in \x{
    \draw [step=1.0,black] (0,1-\xi) grid (\i,-\xi);}
\node at (0.5,-0.5){$\N$};
\node at (1.5,-0.5){$\N$};
\node at (2.5,-0.5){$\P$};
\node at (3.5,-0.5){$\N$};
\node at (4.5,-0.5){$\N$};

\node at (0.5,-1.5){$\P$};
\node at (1.5,-1.5){$\N$};
\node at (2.5,-1.5){$\N$};
\node at (3.5,-1.5){$\N$};
\node at (4.5,-1.5){$\P$};

\node at (0.5,-2.5){$\N$};
\node at (1.5,-2.5){$\N$};
\node at (2.5,-2.5){$\P$};

\node at (0.5,-3.5){$\N$};
\node at (1.5,-3.5){$\P$};

\node at (0.5,-4.5){$\P$};

\end{scope}
\end{tikzpicture}
}

\begin{document}

\begin{frontmatter}

\title{Impartial Chess on Integer Partitions}

\author[inst1]{Eric Gottlieb}
\ead{gottlieb@rhodes.edu}
\affiliation[inst1]{organization={Rhodes College}, city={Memphis},
            state={Tennessee},
            country={ U.S.A.}}

\author[inst2]{Matjaž Krnc}
\ead{matjaz.krnc@upr.si}
\author[inst2]{Peter Muršič}
\ead{peter.mursic@famnit.upr.si}

\affiliation[inst2]{organization={University of Primorska}, city={Koper},
            country={Slovenia}
            }

\begin{abstract}
 Berlekamp proposed a class of impartial combinatorial games based on the moves of chess pieces on rectangular boards. 
 We generalize impartial chess games by playing them on Young diagrams and obtain results about winning and losing positions and Sprague-Grundy values for all chess pieces. 
 We classify these games, and their restrictions to sets of partitions known as rectangles, staircases, and general staircases, according to the approach of Conway, later extended by Gurvich and Ho. 
 The games \textsc{Rook} and \textsc{Queen} restricted to rectangles are known to have the same game tree as $2$-pile \textsc{Nim} and \textsc{Wythoff}, respectively, so our work generalizes these well-known games. 
\end{abstract}

\begin{keyword}
impartial game \sep chess \sep Young diagram \sep integer partition \sep combinatorial game

\MSC[2020] 
91A46 %
\sep 05A17 %

\end{keyword}

\end{frontmatter}

\vfill 
\begin{center}
\textit{In appreciation of Elwyn Berlekamp.}
\end{center}
\vfill
\vfill
\newpage

\section{Introduction}

In a series of YouTube videos, Berlekamp \cite{impchess, impchess1} introduced Impartial chess, a family of impartial combinatorial games based on the moves of various chess pieces that are played on rectangular boards. In this paper, we formalize and extend the study of these games by playing them on Young diagrams of integer partitions. The extension to non-rectangular boards is inspired by previous work on LCTR \cite{Gottlieb2022LCTR}, the mis\`ere form of which, known as \textsc{Downright}, is similar in spirit to Berlekamp's impartial chess. 

Conway \cite{Con76} classified impartial games using certain properties that depend on the interaction of normal and misere play. Gurvich and Ho \cite{GURVICH201854} subsequently undertook a similar effort and showed some of their categories to be equivalent to those of Conway. We refer to the unification of these schemes as the \emph{Conway-Gurvich-Ho (CGH) classification}.  
The CGH classification of \textsc{Nim}, \textsc{Subtraction}, \textsc{Wythoff}, \textsc{Mark}, \textsc{LCTR}, and \textsc{Downright} are established in \cite{Gottlieb2022LCTR,Con76,GURVICH201854}. 
We determine Sprague-Grundy values or winning and losing positions for impartial chess games of all chess pieces on the sets of partitions known as rectangles, staircases, and general staircases.
Some of the games we study occupy the same region of the CGH classification scheme as one of the above games.
Played on a particular set of partitions, we show that \textsc{Rook} occupies a region that was not previously known to contain any game, while \textsc{Queen} occupies a region that was only known to contain a single small artificial game.

This paper is structured as follows. In \cref{sec:prelims}, we provide background on partitions, Young diagrams, impartial games, and Sprague-Grundy theory. In \cref{sec:impchess}, we describe the impartial chess games {\sc Downright}, {\sc Pawn}, {\sc Knight}, {\sc Bishop}, {\sc King}, {\sc Queen}, and {\sc Rook}. In \cref{sec:dag}, we use directed acyclic graphs associated to games to define partition-equivalence and game-equivalence of impartial chess games. In \cref{sec:pawnknight}, we discuss partition- and game-equivalence of \textsc{Pawn}, \textsc{Knight}, and \textsc{Downright}. In \cref{sec:king}, we establish Sprague-Grundy values for \textsc{King} on certain classes of partitions. We characterize losing positions for \textsc{Rook} in \cref{sec:rook} and explain why such characterization for \textsc{Queen} is likely to be difficult in \cref{sec:queen}. In \cref{sec:misere} we take two different approaches to the study of the mis\`ere forms of impartial chess games. The first involves an operation on directed acyclic graphs known as truncation. The second concerns a classification scheme, which we refer to as CGH classification, based on the interaction between normal and mis\`ere forms of a given game. In \cref{sec:questions} we offer some directions for further investigation. 

\section{Preliminaries}\label{sec:prelims}
In this section, we provide background information on partitions, \Youngs{}, impartial games, Sprague-Grundy theory, and complexity theory. To simplify expressions throughout the paper we denote the set of integers, nonnegative integers, and positive integers by $\mathbb Z, \mathbb Z^{\geq 0}$, and $\mathbb Z^{>0}$, respectively. For integers $a$ and $b$ with $a \leq b$, we denote the set $\{a, a+1, \ldots, b\}$ by $[a, b]$ and denote $[1,b]$ by $[b]$. We define $[a, b] = \emptyset$ if $a > b$.

A \emph{directed acyclic graph} (DAG) is a directed graph $G=(V, E)$ which does not contain a directed cycle as a subgraph (for standard graph-theoretic definitions see e.g. \cite{diestel2024graph}).
The length of the longest directed path in $G$ is denoted by $\lp(G)$.

\subsection{\Youngs{} and integer partitions}\label{sec:partitions}

Let $n$ be a nonnegative integer. If $n = \lambda_1 + \cdots + \lambda_r$, where $r$ is a nonnegative integer, each $\lambda_j$ is an integer, 
and $\lambda_1 \geq \cdots \geq \lambda_r > 0$, then we say that $\lambda = \br{\lambda_1, \dots, \lambda_r}$ is a \emph{partition} of $n$. 
The unique partition of 0 is $\br{}$. 
The $\lambda_j$'s are called the \emph{parts} of $\lambda$. 
When some of the parts of a partition are equal, we may abbreviate $\br{\lambda_1^{m_1}, \dots, \lambda_k^{m_k}}$, where $\lambda_j$ appears $m_j\ge0$ times. 

Let $\Y$ denote the set of all partitions of nonnegative integers and let $\Y^+=\Y \setminus \{\br{}\}$. Let $\lambda = \br{\lambda_1, \lambda_2, \ldots, \lambda_r}$ and $\beta = \br{\beta_1, \ldots, \beta_m}$ be partitions. 
We define a partial order $\leq$ on $\mathbb Y$ by 
\[
\lambda \leq \beta\, \iff  r \leq m \text{ and } \lambda_i \leq \beta_i
\]
 for all $i\in[r]$
The relation $\leq$ defines the \emph{Young lattice}; see \cite{Kreweras1965,stanley1999enumerative}.

The \emph{\Young{}} of $\lambda$ is a left-justified array of square cells in which the $i$th row from the top contains $\lambda_i$ cells. No distinction is made between a partition and its \Young{}. 
The \emph{conjugate} partition $\lambda'$ of $\lambda\in \Y^+$ is defined to be $\lambda'=\br{\lambda'_1,\ldots,\lambda'_{\lambda_1}}$ 
where $\lambda'_j$ is the largest integer such that $\lambda_{\lambda'_j}\geq j$; see \cref{fig:ferrer}. 
Thus, the $j$th column of $\lambda$ contains $\lambda'_j$ cells. 
The \Young{} of $\br{}$ has no cells and $\br{}' = \br{}$. 

\begin{figure}[h]
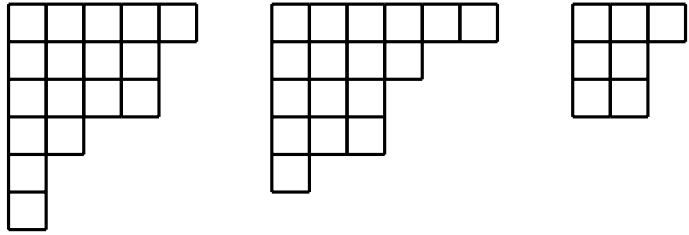

\centering
\picericferrer
\caption{The \Youngs{} of $\lambda=\br{5, 4^2, 2, 1^2}$ (left), its conjugate $\lambda'=\br{6, 4, 3^2, 1}$ (middle), and the subpartition $\lambda[0,2]=\lambda'[1,1]=\br{3,2,2}$ (right).} \label{fig:ferrer}
\end{figure}

Let $\lambda\in\Y^+$, and $i,j \in \mathbb Z^{\geq 0}$.
The \emph{subpartition} $\lambda[i,j]$  of $\lambda$ is defined as $\lambda[i, j] = \br{\lambda_{i+1} - j, \lambda_{i+2} - j, \dots}$, where trailing nonpositive entries are removed. 
Thus $\lambda[0, 0] = \lambda$ for any $\lambda\in \mathbb Y^+$. If $\lambda[i,j] \neq \br{}$, we say that such subpartition is \emph{well-defined}. We define a \emph{corner} of a nonempty partition $\lambda$ to be any pair of integers $(i,j)$ such that  $\lambda[i,j]=\br{1}$. 

Now let $\lambda=\br{\lambda_1,\dots,\lambda_r}> \br1$. Define a partition $\lambda^-$ to be the partition obtained from $\lambda$ by removing all of its corners. 
 More formally, 
 \begin{align*}
    \lambda^-=
    \begin{cases}
    \br{\lambda_1^-,\dots,\lambda_r^-} & \text{if }\lambda_r^-\neq 0 ,\\
    \br{\lambda_1^-,\dots,\lambda_{r-1}^-}& \text{otherwise}
    \end{cases}
\end{align*}
 where 
\[
\lambda_i^-=
\begin{cases}
    \lambda_i-1 & \text{if } i=r \text{ or } \lambda_{i}>\lambda_{i+1},\\
    \lambda_i & \text{otherwise}
\end{cases}
\]
for any $i\in[r]$.
For any $i,j$ such that $\lambda[i,j]>\br{1}$, we have $\lambda[i,j]^{-}=\lambda^{-}[i,j]$.

For a nonempty partition $\lambda = \br{\lambda_1, \ldots, \lambda_r}$ and a positive integer $j \leq \lambda_1$ we define the \emph{$j$th column of $\lambda$} to be the set of pairs $(i, j)$ for which the subpartitions 
$\{\lambda[i-1,j-1]\mid i\in \mathbb Z^{>0}\}$ are well-defined. These pairs graphically match the cells in the the $j$th column of the Young diagram of $\lambda$.
Similarly, for a positive integer $i\le r$ we define the \emph{$i$th row of $\lambda$} to be the set 
of pairs $(i, j)$ for which the subpartitions 
$
\{\lambda[i-1,j-1]\mid j \in \mathbb Z^{>0}\}
$ are well-defined. 

For all non-negative integers $i_1,i_2,j_1,j_2$ such that $\lambda[i_1 + i_2, j_1 + j_2]$ is well-defined, we have the identity  
\begin{align}\label{eq:chainsubpartition}
(\lambda[i_1,j_1])[i_2,j_2]=\lambda[i_1+i_2,j_1+j_2].
\end{align}
For a partition $\lambda=\br{\lambda_1, \dots, \lambda_r}$ the \emph{Durfee length} is defined  to be
$
\D(\lambda)=\max \{\ell \mid \lambda_\ell\ge \ell \}.
$
The \emph{(Dyson's) rank} of $\lambda$ is defined to be $|\lambda_1-r|$, see \cite{dyson1944some}. 

\begin{definition}[Partition Families]\label{def:partition-families}
Given positive integers $r, c$, and $k$, the  \emph{rectangle}, \emph{staircase}, \emph{hook}, and \emph{generalized staircase} partitions are defined respectively as 
\begin{align*}
    \frectangle{r}{c}&= \br{c^r}, 
    &\fstaircase{k}&=\br{k,k-1,k-2,\ldots, 1}, \\
    \fgama{r}{c}&= \br{c, 1^{r-1}},
    &\fgenstaircase rck&=\br{ (kc)^r, ((k-1)c)^r, \ldots, c^r }.
\end{align*} 
\end{definition}
\noindent
See \cref{fig:symgames} for examples. These definitions imply $\fgenstaircase 11k = \fstaircase k$
and
$\fgenstaircase rc1 = \frectangle rc$. 

\begin{figure}
    \centering
 \picturefamilies
    \caption{%
    Partitions corresponding to $\protect\fstaircase{6}$,  $\protect\frectangle64$, $\protect\fgama{6}{4}$ and
    $\protect\fgenstaircase{3}{2}{2}$.
    }
    \label{fig:symgames}
\end{figure}

\noindent The parts of generalized staircases are easily computed from $r, c$, and $k$. 
\begin{observation}\label{obs:row_length}
Let $\lambda=\gs rck=\br{\lambda_1,\ldots,\lambda_{rk}}$ be any generalized staircase and let $i\in \{1,\ldots,rk\}$. 
Then  
\[\lambda_{i}=c \left( k-\floor{\frac{i-1}{r}} \right).\]
\end{observation}
\subsection{Impartial games and Sprague-Grundy values}
\label{sec:imgames}\label{def:PN-recursion}

In this paper, our focus will be exclusively on finite impartial combinatorial games, which we will refer to simply as \emph{games}. A game consists of a set of rules and a (starting) position. If the rules are clear from context we shall not differentiate between the terms position and game. The positions in the games we study are partitions or \Youngs{}, which we describe in the following subsection. A \emph{subposition} of a game is a position that can be reached after a nonnegative number of 
moves. A combinatorial game is \emph{short} if no position can be repeated, and if there exist finitely many subpositions. 

A \emph{terminal position} of a game is a position from which no further moves are possible. 
Under normal play, the player who moves to a terminal position is the winner. Under misère play, the player who moves to a terminal position is the loser. In this paper, we assume normal play unless otherwise indicated.

Let $G$ be a game under normal (resp. misère) play. 
We define the set of its winning positions $\N_G$, 
and the set of its losing positions $\P_G$, as follows:
\begin{enumerate}
    \item Terminal positions are in $\P_G$ (resp. $\N_G$). \label{def:PN-recursion1}
    \item A position is in $\P_G$ if none of its moves leads to a position in $\P_G$.
    \label{def:PN-recursion2}
    \item A position is in $\N_G$ if there is a move to a position in $\P_G$.
    \label{def:PN-recursion3}
\end{enumerate}
We denote $\P_G$ by $\P$ and $\N_G$ by $\N$ when the game is clear from context.

A short two player combinatorial game is \emph{impartial} if both players have the same possible moves in each position.
Let $p$ be a position of such a game $G$. 
If there is a move from $p$ to position $p'$, we write $p \to p'$. 
For a finite set $S$ of nonnegative integers the {\it minimum excluded value} of $S$
is defined as the smallest nonnegative integer that is not in $S$ and
is denoted by  $\Mex (S)$. 
The \emph{Sprague-Grundy} value ($\G$-value) of a position $A$ within the game $G$ is defined recursively by $\G_G(p) = \Mex (\{\G_G(p') \mid p \to p'\}).$
When $G$ is clear from context we denote $\G_G(p)$ by $\G(p)$.
If $\G(p) =t$ then $p$ is called a {\it $t$-position}. 
If $p$ is terminal then $\G(p) = \Mex(\emptyset)=0$.

As introduced in Conway~\cite{Con76}, the Misère Grundy value ($\mathcal G^-$-value) of a position $A$ within the game $G$ is defined recursively the same way as the $\G$-value of $A$, with distinction of terminal positions having $\mathcal G^-$-value $1$.
A position with $\G$-value $t$ and $\mathcal G^-$-value $\ell$ is called a {\it $(t,\ell)$-position}. The pairs $(t, \ell)$ are used to classify combinatorial games as described in \cite{Con76} and \cite{GURVICH201854}. 

$\G$-values provide a useful mechanism for determining winning and losing positions for short impartial games under normal play.
In particular, the $\P$-positions are precisely those that have $\G$-value $0$ while the $\N$-positions are those with nonzero $\G$-values. The Sprague-Grundy  theorem \cite{Sie13} shows that the $\G$-value of the (disjunctive) sum of impartial games under normal play can be obtained by replacing every game-summand with a {\sc Nim}-pile, with the size corresponding to the $\G$-value of the game. For detailed coverage of combinatorial game theory we refer the reader to \cite{Sie13}.

\section{Impartial chess} 

We now describe the family of games played on Young diagrams and inspired by the moves of chess pieces. 
We begin by defining movesets that describe the allowable moves for each chess piece. 

\subsection{Impartial chess games}\label{sec:impchess}

In \cite{Gottlieb2022LCTR}, we studied the mis\`ere version of \textsc{LCTR} and showed that it is equivalent to a normal-form game played on partitions which we refer to as \textsc{Downright}. \textsc{Downright} is closely related to a family of games played on rectangles introduced by Berlekamp \cite{berlekamp2017winning1}. All of these games can be played on partitions and expressed in the language of movesets, which we now explain. 

An \emph{impartial chess} game is a pair $(M, \lambda)$, where $M\subseteq \mathbb Z^{\geq 0}\times \mathbb Z^{ \geq 0} \setminus \{(0,0)\}$ and $\lambda \in \mathbb Y^+$. We refer to $M$ as a \emph{moveset} and an element of $M$ as a \emph{move}. 
On their turn, the players choose a move $(i,j) \in M$, such that $\lambda[i,j]\in Y^+$ and move to $\lambda[i,j]$, passing the turn. 
The above implies that impartial chess games are  finite and short\footnote{For a definition of finite and short impartial game we refer the reader to \cite{Con76}.}.
In addition, we define $(i,j)^+$ as the set 
$\{(ki, kj)\}_{k\in \mathbb Z^{>0}}$.
The impartial chess games presented below, which we study in sections \cref{sec:king,sec:pawnknight,sec:rook,sec:queen}, differ in the choice of moveset. 

\begin{definition}
\label{def:chesspieces}
We study the following impartial chess games. 

\begin{itemize}
    \item {\sc Downright} is  played on the moveset $\Dr = \{(0,1),(1,0)\}$. 
    \item {\sc Pawn} is  played on the moveset $\Pawn = \{(0,1),(1,1)\}$.    
    \item {\sc Knight} is  played on the moveset $\Knight = \{(1,2),(2,1)\}$.
    \item {\sc Bishop} is  played on the moveset $\Bishop = (1,1)^+$.
    \item {\sc King} is  played on the moveset $\K = \{(0,1),(1,0),(1,1)\}$.
    \item {\sc Rook} is  played on the moveset $\R = (1,0)^+ \cup (0,1)^+$.
    \item {\sc Queen} is  played on the moveset $\Q = (1,1)^+ \cup (1,0)^+\cup (1,1)^+$.
\end{itemize}
\end{definition}

Let $(M,\lambda)$ be an impartial chess game and let $(i,j) \in M$ be such that $\lambda[i,j]$ is well-defined. To simplify notation we write $(M, \lambda)[i,j] = (M, \lambda[i,j])$ and $\SG_{(M,\lambda)}(\lambda[i,j])=\SG_M(\lambda[i,j])$.
Furthermore, define $M^{-1}=\{(a,b)\mid (b,a)\in M\}$,  and let   $(M,P)=\{(M,\lambda)\mid \lambda \in P\}$, for any $P\in \Y^+$.

Impartial chess games conveniently allow for easy visualization of $\P/\N$-values and Sprague-Grundy values of subpartitions of a given partition. Specifically, we can write the $\P/\N$-value, or the Sprague-Grundy value, of $\lambda[i, j]$ in the 
cell corresponding to $(i+1)$th row and $(j+1)$th column of $\lambda$, as in \cref{fig:sgyoung}. 
We will say, e.g., the $i$th row of $\lambda$ contains $\P$ when $\lambda[i-1,j] \in \P$ for some $j$. 

\begin{figure}
    \centering
    \picsgyoung
    \caption{Visual representation of Sprague-Grundy values (left), and $\P$/$\N$ value (right) for $(\K,\lambda)$ with $\lambda\le \br{5,5,3,2,1}$.}
    \label{fig:sgyoung}
\end{figure}

\subsection{Directed acyclic graphs, partition-equivalence, and game-equivalence\label{sec:dag}}
Given an impartial chess game $(M,\lambda)$ we define a directed acyclic graph $\DAG{M}{\lambda}$ recursively as follows.
$(0,0)$ is a node of $\DAG{M}{\lambda}$, and if $(i,j)$ is a node of $\DAG{M}{\lambda}$ and $(y,x) \in M$ and $\lambda[i+y,j+x]$ is well defined, then $(i+y, j+x)$ is a node of $\DAG{M}{\lambda}$ and $(i,j) \to (i+y, j+x)$ is an edge of $\DAG{M}{\lambda}$ with label $(y,x)$. 
\begin{figure}
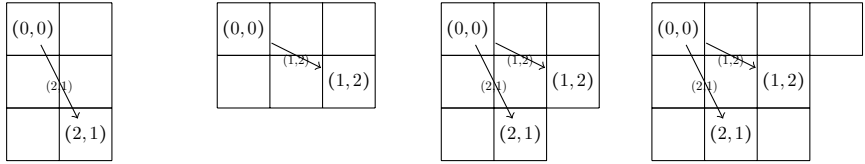
 
    \centering
    \subcaptionbox{$\br{2,2,2}$ and $\br{3,3}$ are game-equivalent but not partition-equivalent in $\Knight$.\label{fig:piclabelequiv}}[.45\textwidth]{\piclabelnotequiv}
    $\quad$
    \subcaptionbox{{$\br{3,3,2}$ and $\br{4,3,3}$ are game-equivalent and partition-equivalent in $\Knight$.}\label{fig:piclabelnotequiv}}[.45\textwidth]{\piclabelequiv}
    \caption{Partition equivalence depends on equality, not isomorphism, of DAGs.}
    \label{fig:piclabelequivyesno}
\end{figure}

For a fixed moveset $M$, we define an equivalence relation on $\Y^+$ in which two partitions $\lambda$ and $\mu$ are equivalent when $\DAG{M}{\lambda} = \DAG{M}{\mu}$. In this case, we say that $\lambda$ and $\mu$ are \textit{partition-equivalent} in $M$. For example, $\br{3,3,2}$ and $\br{4,3,3}$ are partition-equivalent in $\Knight$ because $\DAG{\knight}{\br{3,3,2}} = \DAG{\knight}{\br{4,3,3}}$; see \cref{fig:piclabelequiv}. On the other hand, $\br{3,3}$ and $\br{2,2,2}$ are not partition-equivalent in $\Knight$, even though the underlying DAGs are isomorphic, since $\DAG{\knight}{\br{3,3}} \neq \DAG{\knight}{\br{2,2,2}}$ because their vertex sets and edge labels are different; see \cref{fig:piclabelnotequiv}. 

\begin{observation}\label{obs:dagsubgraph}
    Let $M$ be a moveset and $\lambda,\beta\in \mathbb Y$. 
    Then $\lambda\leq\beta$ implies $\DAG{M}{\lambda}$ is an induced subgraph of $\DAG{M}{\beta}$.
\end{observation}
The set of partitions that are partition-equivalent to $\lambda$ in $M$ has a minimum element (in the sense of Young's lattice). For example, $\br{3,3,2}$ is the minimum element for its equivalence class in $\Knight$, which is 
\[\{\lambda \in \Y^+ : \lambda[1,2] \neq \br{} \mbox{ and } \lambda[2,1] \neq \br{} \mbox{ and } \lambda[2,4] = \lambda[3,3] = \lambda[4,2] = \br{}\}.\] 

For a moveset $M$ and $\lambda \in \Y^+$ it may happen that there is $m=(y,x) \in M$ such that $\lambda[y,x]$ is not well-defined and therefore $m$ can never be used as a move in $(M, \lambda)$. It follows that $m$ does not appear as an edge label in $\DAG{M}{\lambda}$. Let $M|_{\lambda}$ be the set of edge labels in $\DAG{M}{\lambda}$. 

We now define an equivalence relation on the set of impartial chess games. We say that $(M_1, \lambda)$ and $(M_2, \mu)$ are \emph{game-equivalent} if there exists a bijection $f:M_1|_{\lambda} \to M_2|_{\mu}$ such that the following recursive procedure transforms $\DAG{M_1}{\lambda}$ into $\DAG{M_2}{\mu}$. 

\begin{enumerate}
    \item Mark $(0,0)$.
    \item For any edge from a marked vertex $(i, j)$ to an unmarked vertex $(i', j')$ with edge label $m$, replace $(i', j')$ with $(i, j) + f(m)$ and mark it. 
    \item Repeat Item 2 until all nodes are marked. 
    \item Replace each edge label $m$ with $f(m)$. 
\end{enumerate}

\begin{observation}
    If $(M_1, \lambda)$ is game-equivalent to $(M_2, \mu)$, then move $m$ is winning in $(M_1, \lambda)$ if and only if move $f(m)$ is winning in $(M_2, \mu)$. Also, \[\SG_{M_1}(\lambda)=\SG_{M_2}(\mu).\] 
\end{observation}

\begin{observation}[Conjugate invariance]\label{obs:conjugateinvariance}
An impartial chess game $(\lambda,M)$ is game-equivalent to $(\lambda',M^{-1})$.
\end{observation}

The following observation follows directly from \cref{obs:dagsubgraph}.
\begin{observation}\label{obs:height}
For any moveset $M$ and  $\lambda,\beta\in \Y^+$ such that $\lambda \leq \beta$, we have
$\lp(\DAG{M}{\lambda}) \leq \lp(\DAG{M}{\beta})$.
\end{observation}

\begin{proposition}\label{prop:SGdiam}
   For any $(M, \lambda)$ we have $\SG_M(\lambda)\leq \lp(\DAG{M}{\lambda})$.
\end{proposition}

\begin{proof}
We proceed by induction on $k=\lp(\DAG{M}{\lambda})$.
If $k=0$ then the claim follows since $\SG_M(\lambda) = 0$, so assume $k>0$. Define $S=\left\{\bar\lambda \mid (M,\lambda)\to \left(M,\bar \lambda \right)\right\}$, and notice that for any $\bar\lambda\in S$ we have  $\lp\left(\DAG{M}{\bar\lambda}\right)<k$. 
By induction,
\[\SG_M(\lambda)=\Mex\left\{\SG_M\left(\bar\lambda\right) : \bar\lambda \in S\right\} \le \Mex\{0, \ldots, k-1\} = k. \qedhere \]
\end{proof}

    \subsection{\textsc{Pawn}, \textsc{Knight}, and \textsc{Downright}}
\label{sec:pawnknight}

In this section, we show that \textsc{Pawn} (with moveset $\Pawn = \{(0,1),(1,1)\}$, see \cref{def:chesspieces}) and \textsc{Knight} (with moveset $\Knight = \{(1,2),(2,1)\}$)  can both be reduced to \textsc{Downright} (with moveset $\Dr = \{(0,1),(1,0)\}$), which was studied in \cite{Gottlieb2022LCTR}. 
For all three games, a particular mirroring strategy may be used to obtain the following observation.

\begin{proposition}
For any $\lambda\in \Y$, the following are true, provided the respective subpartitions are well-defined:
\begin{align*}
\SG_\Dr(\lambda)&=\SG_\Dr(\lambda[1,1]), & 
\SG_\Pawn(\lambda)&=\SG_\Pawn(\lambda[2,1]), & 
\SG_\Knight(\lambda)&=\SG_\Knight(\lambda[3,3])
\end{align*}
\end{proposition}
\begin{proof}
    The nim-sum of two identical games is always in $\P$.
    Now notice that $(\Dr,\lambda) \oplus (\Dr,\lambda[1,1])$ and $(\Pawn,\lambda) \oplus (\Pawn,\lambda[2,1])$ and $(\Knight,\lambda) \oplus(\Knight,\lambda[3,3])$ are in $\P$, from which the result follows.
\end{proof}

\noindent In the remainder of \cref{sec:pawnknight} we provide a fuller treatment of $\Dr,\Pawn$ and $\Knight$.

\subsubsection*{ \textsc{Pawn} is solved by \textsc{Downright}}
\label{sec:pawn}

For  $\lambda=\br{\lambda_1,\dots,\lambda_r} \in \mathbb Y^+$ define transformation $\varphi_\Pawn(\lambda)=
\br{\lambda_1,\lambda_2-1,\dots,\lambda_r-r+1},$
omitting nonpositive entries. 
Notice that $(\Pawn,\lambda)$ is game-equivalent to $(\Dr,\varphi_\Pawn(\lambda))$ for every  $\lambda$. In particular, 
$\SG_\Pawn(\lambda)=\SG_\Dr(\varphi_\Pawn(\lambda))$. See \cref{fig:reduceknight}. 

Not every game of \textsc{Downright} is game-equivalent to some game of \textsc{Pawn}.
For example, there is no partition $\mu$ such that $(\Dr,\br{2,2})$ is game-equivalent to $(\Pawn, \mu)$. 
It is easy to verify that
$(\Dr,\lambda)$ is game-equivalent to some game of $\Pawn$ if and only if $\br{2}$ and $\br{1,1}$ are not subpartitions of $\lambda$.

\subsubsection*{ \textsc{Knight} is solved by \textsc{Downright}}

For  $\lambda=\br{\lambda_1,\dots,\lambda_r} \in \mathbb Y$ define transformation $\varphi_\Knight(\lambda)=\bar\lambda$  such that $\bar\lambda[i,j]$ is well-defined if and only if $\lambda[2i+j,2j+i]$ is. In other words, $\varphi_\Knight(\lambda) = (\bar \lambda_1, \bar \lambda_2, \ldots) \text{ where } \bar \lambda_j = \max\left( \{t \in [r] \, : \, 2t + j - 2 \leq \lambda_{t + 2j - 2}\} \cup \{0\} \right).$ Verify that $(\Knight,\lambda)$ is game-equivalent to $(\Dr,\varphi_\Knight(\lambda))$ for every $\lambda \in \Y$ and that 
$\SG_\Knight(\lambda)=\SG_\Dr(\varphi_\Knight(\lambda)).$ See \cref{fig:reduceknight}. 

Not every game of \textsc{Downright} is game-equivalent to a game of \textsc{Knight}. 
For example, there is no partition $\mu$ such that $(\Dr,\br{3})$ is game-equivalent to $(\Knight, \mu)$. 
It is easy to verify that
$(\Dr,\lambda)$ is game-equivalent to some game of $\Knight$ if and only if $\br{3}$ is not a subpartition of $\lambda$ and $\br{1,1,1}$ is not a subpartition of $\lambda$.

\begin{figure}[!h]
    \centering
    \knightmoves
    \caption[fragile]{\textsc{Knight} and \textsc{Pawn} on $\lambda=\br{14^2,13,12^4,9,5,3}$, shaded as \!\!\knightsquare and \!\!\pawnsquare\!\!, are game-equivalent to \textsc{Downright} on 
    $\varphi_\Knight(\lambda)=\br{6,5,4,2,1}$ and  $\varphi_{\Pawn}(\lambda)=\br{14,13,11,9,8,7,6,2}$, resp.}
    \label{fig:reduceknight}
\end{figure}
    
\subsection{\textsc{King}}
\label{sec:king}
In the video \cite{impchess2}, Berlekamp establishes the following result. We give a written proof here for completeness.
\begin{theorem}[Berlekamp]\label{thm:SGKingrectangle} Let $r,c$ be positive integers. Then
\[\SG_\K\left( \frectangle{r}{c} \right) = \begin{cases}
0 & \text{if $c$ and $r$ are odd}  \\
2 & \text{if $c$ and $r$ are even} \\
1 & \text{if $c+r$ is odd and  $\min(c,r)$ is odd} \\
3 & \text{if $c+r$ is odd and  $\min(c,r)$ is even}
\end{cases}
\]
\end{theorem}

\begin{proof}[Proof of \cref{thm:SGKingrectangle}]
We proceed by induction on $r+c$. 
For the base case, suppose $r=1$ or $c=1$. 
The case $r=1=c$ is trivial.
If $r=1$ and $c>1$, then there is only one permissible move, i.e. $(0,1)$. Thus the $\SG_{\K}$ values must alternate between $0$, when $c$ is odd, and $1$, when $c$ is even, as stated.
The case when $r>1$ and $c=1$ is similar.

Consider a position $\lambda=\frectangle{r}{c}$ where $r>1$ and $c > 1$. Such a position has three available moves $\lambda[1,0]=\frectangle{r-1}{c}$ and $\lambda[1,1]=\frectangle{r}{c-1}$ and $\lambda[0,1]=\frectangle{r-1}{c-1}$. 
We consider the cases as in the statement of the claim.
\begin{description} 
    \item[] If $c$ and $r$ are odd:
    \begin{description}
    \item By induction $\SG_{\K}(\lambda[1,1])=2$ so
    $\left\{ \SG_{\K}(\lambda[1,0]),\SG_{\K}(\lambda[0,1]) \right\} \subseteq \{1,3\}$.
    \item Thus 
    $\SG_{\K}(\lambda)=\Mex\left( \SG_{\K}(\lambda[0,1]),\SG_{\K}( \lambda[1,1]),\SG_{\K}(\lambda[1,0])\right)=0.$
    \end{description}    
    
    \item[]  If $c$ and $r$ are even:
    \begin{description}
        \item By induction $\SG_{\K}(\lambda[1,1])=0$. Then $\min(c-1,r)$ is odd or 
        \item $\min(c,r-1)$ is odd so $1 \in \left\{ \SG_{\K}(\lambda[1,0]),\SG_{\K}(\lambda[0,1]) \right\} \subseteq \{1,3\}.$
        \item Thus
    $ \SG_{\K}(\lambda)=\Mex\left( \SG_{\K}(\lambda[0,1] ),\SG_{\K}(\lambda[1,1]),\SG_{\K}(\lambda[1,0]) \right)=2.$
    \end{description}

    \item[] If $c+r$ is odd and  $\min(c,r)$ is odd:
    \begin{description}
        \item[] 
    We have $\SG_{\K}(\lambda[1,1])=3$ and
    $\{\SG_{\K}(\lambda[1,0]),\SG_{\K}(\lambda[0,1])\} = \{0,2\}$. 
    \item Thus 
    $\SG_{\K}(\lambda)=\Mex(\SG_{\K}(\lambda[0,1]),\SG_{\K}(\lambda[1,1]),\SG_{\K}(\lambda[1,0]))=1.$
    \end{description}
    
    \item[] If $c+r$ is odd and  $\min(c,r)$ is even:
    \begin{description}
        \item We have $\SG_{\K}(\lambda[1,1])=1$ and
    $\{\SG_{\K}(\lambda[1,0]),\SG_{\K}(\lambda[0,1])\} = \{0,2\}$.
    \item So 
    $\SG_{\K}(\lambda)=\Mex(\SG_{\K}(\lambda[0,1]),\SG_{\K}(\lambda[1,1]),\SG_{\K}(\lambda[1,0]))=3.\qedhere$ 
    \end{description}
\end{description}
\end{proof}
\noindent We continue with sufficient criteria for certain partitions to be winning.

\begin{proposition}\label{prop:pnrcevenstronger}
Let $\lambda$ be a nonempty partition with all parts
\begin{enumerate}
    \item appearing an even number of times. Then $\lambda[i,j] \in \N_{\K}$ for all even $i$.
    \item being even. Then $\lambda[i,j] \in \N_{\K}$ for even $j$.
\end{enumerate}
\end{proposition} 

\begin{proof}
    We prove only the first case, as the second follows by the conjugate invariance (see \cref{obs:conjugateinvariance}) of $\K$. In this case we write $\lambda=\br{\lambda_1^{m_1},\ldots,\lambda_k^{m_k}}$ where $m_t$ is even for all $t\in[k]$.
    We proceed by induction on the number of parts $m_1 + \cdots + m_k$. The base case $\lambda=\br{\lambda_1^2}$ is a special case of \cref{thm:SGKingrectangle}.
    Assume the statement holds for all such partitions with fewer parts. For $(\K,\lambda[i,j])$ where $i$ is even and at least two, the statement holds from the induction hypothesis. 
    Thus no $P_{\K}$-position is reachable from $(\K,\lambda[1,j])$ via $(1,1)$- or $(1,0)$-move.
    It follows that $\lambda[i,j]\in P_{\K} \iff \frectangle2{\lambda_1}[i,j]\in P_{\K}$ for all $i\in [0,1]$ and $j\in[0,\lambda_1-1]$. Thus, the result follows from \cref{thm:SGKingrectangle}.
\end{proof}

\begin{corollary}\label{cor:allevenallduplicated}
Let $\lambda=\br{\lambda_1,\ldots,\lambda_r}$ where all parts appear an even number of times. Then $\lambda[i,j] \in \P_{\K}$ if and only if both $i$ and $\lambda_{i+1}-j$ are odd. 
\end{corollary}

\begin{proof}
By \cref{prop:pnrcevenstronger} we know  $\lambda[2k,j] \in \N_{\K}$. For $i$ odd, there is neither a $(1, 0)$-move nor a $(1,1)$-move from $\lambda[i,j]$ to a $P_{\K}$-position. Therefore $\lambda[i,j]$ alternates between $\P_{\K}$ and $\N_{\K}$ positions, starting with $j=\lambda_{i+1}-1$ as $\P_{\K}$. 
\end{proof}

We use \cref{prop:pnrcevenstronger} to derive $\P_\K/\N_\K$-positions of generalized staircases.

\begin{corollary}
If $rc$ is even then $\gs rck \in \N_{\K}$ for any $k$.
\end{corollary}

Combining \cref{obs:row_length} and \cref{cor:allevenallduplicated} yields the following. 
\begin{corollary}\label{cor:PNoneeven}
    Let $r$ be even and $c, k$ be arbitrary. Then 
    $\gs rck[i,j] \in \P_\K$ if and only if $i\left(c \left( k-\left \lfloor \frac{i}{r} \right \rfloor \right )-j \right)$ is odd. 
\end{corollary}

We now turn our attention to the Sprague-Grundy values of \textsc{King}. 

\begin{lemma} Let $k \in \mathbb Z^{\geq0}$. Then $\SG_{\K}(\fstaircase{k})= (k-1) \bmod 3.$
\end{lemma}

\begin{proof}
We proceed by induction on $k$, with base cases for $k=1$ and $k=2$.
$\fstaircase{1}$ is terminal so $\SG_{\K}(\fstaircase{1})=0$. The positions reachable from  $\fstaircase{2}$ both have $\SG$-value 0, so  $\SG_{\K}(\fstaircase{2})=1$. If $k > 2$, then the reachable positions from $\fstaircase{k}$ are  $\fstaircase{k-1}$ and  $\fstaircase{k-2}$. By induction hypothesis, those positions have $\SG$-values $(k-2) \bmod 3$ and $(k-3) \bmod 3$. Thus $\SG_{\K}(\fstaircase{k}) = \Mex(\{(k-2) \bmod 3, (k-3) \bmod 3\}) = (k-1) \bmod 3$ as desired. 
\end{proof}

\begin{lemma}
Let $r$ be even. Then $\SG_\K(\fgenstaircase rrk)=\begin{cases}
2 & \text{if  $k =1$,} \\
1 & \text{otherwise.}
\end{cases}$
\end{lemma}

\begin{proof}
Let $\lambda=\fgenstaircase rrk$.
We use induction on $k$ to prove something stronger. 
In addition the the statement of the lemma we also prove that the sequence $\SG_\K(\lambda[0,j])$, for $0\leq j\leq r-1$, alternates between $1$ and some value in $\{2,3\}$.
By conjugate invariance, the same holds for $\SG_\K(\lambda[j,0])$.

If $k = 1$, then by \cref{thm:SGKingrectangle}, we know that $\SG_\K(\fgenstaircase rr1)=2$ and that $\SG$-values along the top row and left column alternate between 2 and 1. 

Now consider the case $k > 1$. We proceed by iteration. 
\begin{enumerate}[label=$(S_\arabic*)$]
    \item \label{en:s1}By induction hypothesis we have that $\lambda[r,j]$ alternates between $1$ and $\{2,3\}$ for $0 \leq j \leq r-1$ and so does $\lambda[i,r]$ for $0 \leq i \leq r-1$.

    \item For odd $i$ with $0\leq i\leq r-1$, we have that $\lambda[i,r-1]$ and $\lambda[r-1,i]$ are 0-positions. \cref{cor:PNoneeven} implies that $\lambda[i,j]\in \P_\K$  precisely when $i$ and $j$ are both odd, so $\SG_\K(\lambda[i,j]) = 0$ precisely when $i$ and $j$ are odd. 
    \label{en:s2}
    \item For even $i$ with $0\leq i\leq r-1$, we have that $\lambda[i,r-1]$ and $\lambda[r-1,i]$ are $\{2,3\}$-positions, because each of them can move to 
    a $1$-position, a $\{2,3\}$-position, and a $0$-position, due to  \ref{en:s1} and \ref{en:s2}, respectively. 
    \label{en:s3}
\end{enumerate}
Repeating these three steps $r/2$ times on consecutive rows and columns
we get 
\begin{equation}
\label{eq:sgrrk2}
\SG_\K(\lambda[i, j])= \begin{cases}
    0 & \text{ if $i$ and $j$ are both odd,}\\
    1 & \text{ if $i$ and $j$ are both even,}\\
    2 \mbox{ or } 3 & \text{ otherwise.}
\end{cases}
\end{equation}
In particular, \eqref{eq:sgrrk2} implies $\SG_\K(\lambda)=1$ while $\SG_\K(\lambda[0,j])$ alternates between $1$ and some value in $\{2,3\}$ for $0\leq j\leq r-1$. 
\end{proof}

\subsection{\textsc{Rook}}
\label{sec:rook}
In this section, we characterize $\P$-positions of \textsc{Rook} and give $\SG_\R$-values for certain generalized staircases. Recall that 
$\R = (1,0)^+ \cup (0,1)^+$. 
\begin{observation}\label{obs:2pilenim}
\textsc{Rook} on a rectangular partition is game-equivalent to two-pile \textsc{Nim}. 
In particular, if $r$ and $c$ are positive integers, then $\G_\R(\rectangle_{r,c})=(r-1)\oplus(c-1)$.
\end{observation}
\noindent Thus, the results we obtain for non-rectangular boards can be viewed as generalizations of this game. The main result of this section is the characterization of the structure of all $\P$-positions with respect to \textsc{Rook}.

\begin{theorem}
Let $\lambda=\br{\lambda_1,\ldots,\lambda_r}$ be partition distinct from $\br1$.
Then $(\R,\lambda)$ is a losing position if and only if $\lambda$ has rank 0 and $\lambda[1,1] \ge \fstaircase{r-1}$. 
\end{theorem}

We proceed with an important definition specific to this section.
A partition $\lambda \neq \br{}$ is said to be \emph{ample} if for every $i \in [0,r-1]$ there exists $j$ such that $\lambda[i,j] \in \P_\R$, and for every $j \in [0,\lambda_1-1]$ there exists $i$ such that $\lambda[i,j] \in \P_\R$. 

Note that if $\lambda[i, j] \in \P_\R$, then $\lambda[i, k]$ and $\lambda[m, j]$ are in $\N_\R$ for $k \neq j$ and $m \neq i$, provided they are well-defined. 
In other words, each row and column of $\lambda$ contains at most one $\P_\R$-position. Understanding the structure of ample partitions is essential to our theorem, due to the following lemma.

\begin{lemma}\label{lem:rookPchar}
Let $\lambda=\br{\lambda_1,\ldots,\lambda_r}$ be a partition distinct from $\br1$. Then $\lambda \in \P_\R$ if and only if $\lambda$ has rank $0$,
$\lambda_1=\lambda_2$,
$\lambda_r\ge 2$ and
$\lambda[1,1]$ is ample.
\end{lemma}

\begin{proof}
Let $\lambda \in \P_\R$.
By $\P/\N$-recursion it follows that
\begin{align}
   \lambda[0,1],\ ..., \ \lambda[0,\lambda_1-1], \ \lambda[1,0], \ ..., \ \lambda[r-1,0] 
   \label{eq:border_subpositions}
\end{align}
are in $\N_\R$. By $\P/\N$-recursion $\lambda[0,j]$ has a move to a $\P_\R$-position for each $j=1,...,\lambda_1-1$. Since $\lambda[0,j+1], \ ..., \ \lambda[0,\lambda_1-1]$ are $\N_\R$-positions, there is a $\P_\R$-position among $\lambda[1,j], \ ..., \ \lambda[\lambda_j'-1,j]$, where $\lambda'=\br{\lambda'_1,\ldots,\lambda'_{r'}}$ is the conjugate of $\lambda$. In particular, $\lambda[1,j]$ is well-defined, so there is a $\P_\R$-position in column $j$. 

Similarly, there is a move from $\lambda[i,0]$ to a $\P_\R$-position for each $i=1,...,r-1$ so there is a $\P_\R$-position in row $i$, implying  $\lambda[1,1]$ is ample. $\lambda[1,j]$ and $\lambda[i,1]$ are well-defined for $j=1,...,\lambda_1-1$ and $i=1,...,r-1$, so $\lambda_1=\lambda_2$ and $\lambda_r\ge 2$.

For the converse, assume $\lambda$ has rank 0,
$\lambda_1=\lambda_2$,
$\lambda_r\ge 2$ and  $\lambda[1,1]$ is ample. Consider the subpositions \eqref{eq:border_subpositions} reachable from $\lambda$ in one move.
Since $\lambda[1,1]$ is ample, all the subpositions from \eqref{eq:border_subpositions} can reach a $\P_\R$-position in a single move, so are in $\N_\R$.
Thus any move from $\lambda$ is to an $\N_\R$-position, so $\lambda$ is a $\P_\R$-position.
\end{proof}

\begin{lemma}\label{lem:equi-p}
A partition $\lambda=\br{\lambda_1,\ldots,\lambda_r}$ is ample if and only if 
$\lambda$ has rank $0$ and $\lambda \geq \fstaircase{r}$.
\end{lemma}
\begin{proof}
Let $\lambda'=\br{\lambda'_1,\ldots,\lambda'_{r'}}$ be the conjugate of $\lambda$.
\begin{enumerate}
    \item[$\Leftarrow$] \label{enu:left}
    Assume that $\lambda$ has rank $0$ and $\lambda \geq \fstaircase{r}$. We show that for each $i=1,...,r$ we have exactly one $\P_\R$-position among 
    \begin{align}
        \lambda[i-1,0], \ \lambda[i-1,1], \ \ldots, \ \lambda[i-1,\lambda_i-1].\label{eqn:p_position}
    \end{align}
    The proof is by backwards induction on the rows. 
    The statement is true for the last row as moving right as far as possible gives position $\br{1} \in \P_\R$.
    We assume the statement holds for the rows $i+1,\ldots,r$ and show the statement is true for row $i$. Among rows $i+1,\ldots,r$ the number of $\P_\R$-positions is exactly $r-i$ by induction hypothesis. Since $\lambda \geq \fstaircase{r}$ the sequence \eqref{eqn:p_position} is of length $\lambda_i \geq r-i+1$, implying at least one position in the sequence will not be able to move to a $\P_\R$-position, by pigeon-hole principle. Such a position is in $\P_\R$. Since $\lambda$ has rank $0$ and $\lambda \geq \fstaircase{r}$ if and only if $\lambda'$ has rank $0$ and $\lambda' \geq \fstaircase{r}$, it holds that $\lambda'$ has exactly one $\P_\R$-position in every row, thus $\lambda$ has a $\P_\R$-position in every column. Thus $\lambda$ is ample.
    
    \item[$\Rightarrow$] \label{enu:right}
    Assume that $\lambda$ is ample, and let $n_0$ be the number of subpositions of $(\R,\lambda)$ which are in $\P_\R$. As $\lambda$ is ample, this implies that $n_0=r$, as well as $n_0=\lambda_1$, which in turn implies that $\lambda$ has rank $0$. It remains to show that $\lambda \geq \fstaircase{r}$. 
    Suppose not, and let $i$ be the largest integer such that $\lambda_i\leq r-i$.
    Because $\lambda$ is ample, for each $k = i+1, \ldots, r$ there is an $\ell$ such that $\lambda[k, \ell] \in \P_\R$. Since the sequence in \eqref{eqn:p_position} has length $\lambda_i \leq r-i$ and there are no $\ell, k_1, k_2$ with $k_1 \neq k_2$ such that $\lambda[k_1, \ell] \in \P_\R$ and $\lambda[k_2, \ell] \in \P_\R$, it follows that each position in \eqref{eqn:p_position} can move to one of those $\P_\R$-positions, implying that all of them are $\N_\R$.
    Thus row $i$ has no $\P_\R$-positions, contradicting the assumption that $\lambda$ does not contain a $\fstaircase{r}$. \qedhere
\end{enumerate}
\end{proof}

We conclude this subsection with a result concerning the Sprague-Grundy values of \textsc{Rook} on subpositions of $\fgenstaircase {2^\ell}{2^\ell}k$.

\begin{proposition}
For any positive integers $\ell,k$, and any  $i,j$ such that $0 \leq \left\lfloor\frac{i}{2^\ell}\right\rfloor+\left\lfloor\frac{j}{2^\ell}\right\rfloor \leq k-1$,
let $f(i,j)=k-1-\left\lfloor\frac{i}{2^\ell}\right\rfloor-\left\lfloor\frac{j}{2^\ell}\right\rfloor$.
We have 
\begin{align}\label{eq:clfij}
\SG_\R \left ( \fgenstaircase {2^\ell}{2^\ell}k[i,j] \right)= 2^\ell
 f(i,j)+\left( (-i-1 ) \oplus (-j-1)\right) \bmod 2^\ell.
 \end{align}
Furthermore, let $S_\tau=\left \{\SG_\R\left (\fgenstaircase {2^\ell}{2^\ell}k[i,j] \right) \colon f(i,j)= \tau\right\}$ for any integers $\tau \in [0,k-1]$ and $i\in [0,(k-\tau)2^\ell-1]$.
We have
\begin{align*}
S_\tau= \left [\tau2^\ell,(\tau+1)2^\ell-1 \right] \qquad \text{and}
\qquad \bigcup_{i=0}^\tau S_i= \left[0,(\tau+1)2^\ell-1 \right].
\end{align*}

\end{proposition}
\begin{proof}
Observe that  $\fgenstaircase {2^\ell}{2^\ell}k[i,j]$ 
is well-defined if and only if $0 \leq \left\lfloor\frac{i}{2^\ell}\right\rfloor+\left\lfloor\frac{j}{2^\ell}\right\rfloor \leq k-1$.We proceed by induction on $\tau$. 
If $\tau=0$, i.e. for $i,j$ with $f(i,j)=0$, we have
$\fgenstaircase {2^\ell}{2^\ell}k[i,j]=\frectangle{ 2^\ell-(i\bmod 2^\ell)}{2^\ell-(j\bmod 2^\ell)}$ so
\begin{align*}
\SG_\R\left(\fgenstaircase {2^\ell}{2^\ell}k[i,j] \right)&= \left(( 2^l - (i \bmod 2^\ell)-1) \oplus (2^l - (j \bmod 2^\ell)-1\right)\\
&= \left((-i-1) \bmod 2^\ell \right) \oplus \left((-j-1) \bmod 2^\ell \right) \\
&= \left( (-i-1 ) \oplus (-j-1)\right) \bmod 2^\ell \leq 2^\ell -1,
\end{align*}
where the first equality is due to \cref{obs:2pilenim}. Thus 
we have that 
\[\left\{ \SG_\R\left( \fgenstaircase {2^\ell}{2^\ell}k[t,u] \right) \colon f(t,u)=0 \right\}= \left [0,2^\ell-1 \right].\]
Assume now $\tau>0$.
To show the first part of the claim, i.e. \eqref{eq:clfij}, fix arbitrary $i,j$ such that $f(i,j)=\tau$. 
By induction hypothesis all $\G$-values of positions $(i_1,j_1)$ with $f(i_1,j_1)\neq \tau$ exactly correspond to  values from $\bigcup_{i=0}^{\tau-1}S_{i}$. 
This implies that the $\G$-value of $\fgenstaircase {2^\ell}{2^\ell}k[i,j]$ is the same as 
$\SG_\R \left( \frectangle {2^\ell}{2^\ell}[i',j'] \right)$ incremented by $\tau2^\ell$, where $i',j'$ are the remainders of $i,j$, divided by $2^\ell$, respectively. 
Observe that this is exactly what we want to show in \cref{eq:clfij}.

To verify the remainder of the claim it is enough to recall that the $\G$-values of a square $\frectangle{r}{r}$ form a latin square, i.e. use all values from $[0,r-1]$ in every row as well as in every column. 
\end{proof}

\begin{corollary}
    \label{thm:rookgenstairsg}
Let $k,\ell,i,j$ be such that $\fgenstaircase {2^\ell}{2^\ell}k[i,j]$ is well-defined. Then $\SG_\R \left(\fgenstaircase {2^\ell}{2^\ell}k \right)=2^\ell(k-1)$, and more generally, 
      \[\SG_\R \left( \fgenstaircase {2^\ell}{2^\ell}k[i,j] \right) = 2^\ell\left(k-1-\left\lfloor\frac{i}{2^\ell}\right\rfloor-\left\lfloor\frac{j}{2^\ell}\right\rfloor\right)+\left( (-i-1 ) \oplus (-j-1)\right) \bmod 2^\ell.\]
\end{corollary}

\subsection{\textsc{Queen}}
\label{sec:queen}
\textsc{Queen} is played with moveset $\Q=(1,0)^+ \cup (0, 1)^+ \cup (1, 1)^+.$
When played on rectangular partitions, \textsc{Queen} is closely related to \textsc{Wythoff}. 
Recall that \textsc{Wythoff} \cite{Wythoff}, denoted by $W(a,b)$, is played with two piles of tokens of size $a$ and $b$,
where a move consists of removing an arbitrary positive number of tokens from a pile, or removing the same number of tokens from both piles. 
For $c,r \in \mathbb Z^{>0}$, the game $(\Q,c^r)$ is the same as $W(c-1, r-1)$; see Nivasch \cite{nivasch2005more}.

Let $\lambda = \br{\lambda_1, \ldots, \lambda_r}$ be a partition and let $k$ be a positive integer. The next result concerns a ``stair'' of $\lambda$-steps. We denote such a partition by $S_{\lambda}^k$; for example, $S_{\br{3, 1}}^3$ is shown in \cref{fig:lambda-stair}. More formally, let $S_{\lambda}^k$ be the partition 
\[\begin{array}{c}
    \br{(k-1) \lambda_1 + \lambda_1, \ (k-1)\lambda_1 + \lambda_2, \ \ldots, \ (k-1) \lambda_1 + \lambda_r,  \phantom{.}\\
    \phantom{(} (k-2)\lambda_1 + \lambda_1, \ (k-2)\lambda_1 + \lambda_2, \ \ldots, \ (k-2)\lambda_1 + \lambda_r, \phantom{.}\\ 
    \vdots \\
    \phantom{(}(k-k) \lambda_1 + \lambda_1, \ (k-k)\lambda_1 + \lambda_2, \ \ldots, \ (k-k) \lambda_1 + \lambda_r }.
\end{array}\]

\begin{figure}[b]
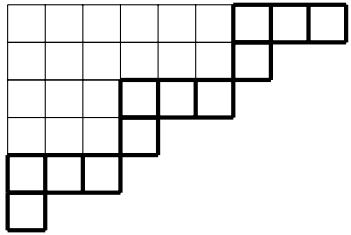

    \centering
    \picslambdak
    \caption{The partition $S_{\br{3, 1}}^3$.}
    \label{fig:lambda-stair}
\end{figure}
\begin{lemma}
Let $\lambda$ be a partition of rank $0$ and let $k>1$ be an odd integer.
Then $\left(\Q,S_\lambda^k \right)\in \N_\Q$.
\end{lemma}

\begin{proof}
First suppose that $\lambda$ is a $\P_\Q$-position. Then the first player can move to $S_{\lambda}^k[(k-1)\lambda_1, 0] \in \P_\Q$. Thus, $S_\lambda^k$ is an $\N_\Q$-position. 

Next, suppose that $\lambda$ is an $\N_\Q$-position. Then there is a move from $\lambda$ to a $\P$-position, either along the first row, the first column, or the main diagonal of $\lambda$. The first player can win on $S_\lambda^k$ by moving to that $\P$-position of the first, middle, or last $\lambda$-step respectively by moving on the first row, first column, or main diagonal of $S_\lambda^k$ respectively (the upper left corner of the middle $\lambda$-step lies on the main diagonal of $S_\lambda^k$ since $\lambda$ has rank $0$). Thus $S_\lambda^k$ is an $\N$-position. 
\end{proof}

\begin{lemma}\label{lem:double}
Let $\lambda = \br{\lambda_1, \dots, \lambda_r}$ be a partition. 
If $\lambda_1-1>2(r-1)$ or $r-1>2(\lambda_1-1)$ then $\lambda \in \N_\Q$-position.
\end{lemma}

\begin{proof}
Suppose $\lambda_1-1>2(r-1)$; the case $r-1>2(\lambda_1-1)$ is similar. If a row contains a $\P_\Q$-position, then every cell in the same row and to the left of it is an $\N_\Q$-position since it is possible to reach the $\P_\Q$-position from such a cell in a single move. Thus, each row contains at most one $\P_\Q$-position. We will show that there is a $\P_\Q$-position in the first row to the right of the first column, making $\lambda$ an $\N_\Q$-position. 

If the last cell in the first row is a $\P_\Q$-position, we're done, so suppose not. The $\N_\Q$-positions in the first row are precisely those cells from which it is possible to reach a $\P_\Q$-position in a lower row. Any cell in a row below the first that is a $\P_\Q$-position can be reached in a single move from one or two cells in the first row (certainly by a vertical move and possibly a diagonal move). Thus the number of $\N_\Q$-positions in the first row and not the first column is at most $2(r-1)$. But the number of cells in the first row and not the first column is $\lambda_1-1 > 2(r-1)$, so at least one of them must be a $\P_\Q$-position, making $\lambda$ an $\N_\Q$-position.  
\end{proof}

We conclude this section with a lemma that allows us to determine the Sprague-Grundy value of a partition in the games of \textsc{Rook} and \textsc{Queen} for partitions contained in a certain staircase. 

\begin{lemma}\label{lem:queenrook}
Let $M\in \{\R,\Queen\}$, 
 let $\lambda=\br{\lambda_1,\dots,\lambda_r}$ be a partition, and let $k=\max(\lambda_1,r)$.
 If $ \lambda \leq \fstaircase{k}$, then $\SG_M(\lambda)=k-1$.
\end{lemma}

\begin{proof}

Suppose without loss of generality that  $\lambda_1\leq r$ so that $k = r$.
We claim that $\lp(\DAG{M}{\lambda})=k-1$.
Indeed, $\lambda \leq \fstaircase{k} $ and \cref{obs:height} imply that  
\[\lp(\DAG{M}{\lambda}) \le \lp(\DAG{M}{\fstaircase{k}})=k-1.\]
Hence, \cref{prop:SGdiam} implies that $\SG_M(\lambda)\le \lp(\DAG{M}{\lambda}) \le k-1$. 
To establish the claim it suffices to show $\SG_M(\lambda)\ge  k-1$. We proceed by induction on $k$. 
For the base case, we have $\SG_M(\br{1}) = 0$. 
For the induction step, we have 
\begin{align*}
\SG_M(\lambda) &\ge \Mex\left(\{\SG_M(\lambda[i,0])\}_{1\le i\le k-1}\right)
= \Mex\{0,\ldots,k-2\} 
= k-1.  \qedhere
\end{align*}
\end{proof}

\section{Misère aspects of impartial chess} \label{sec:misere}

In this section we discuss two aspects of misère play for impartial chess games. In \cref{sec:trunc}, we show how misère play is equivalent to normal play on a ``truncated'' partition using the same moveset. In  \cref{sec:CG}, we use a classification scheme initiated by Conway \cite{Con76} and further developed by Gurvich and Ho \cite{GURVICH201854, Gur11} to study the interaction of normal play and mis\`ere play on impartial chess games.

\subsection{The truncation operation\label{sec:trunc}}

In this section we describe the behavior of misère variants of games studied in \cref{sec:impchess}. 
To this end, we make use of \emph{truncation}, introduced in \cite{Gottlieb2022LCTR}. This notion, reformulated in the setting of impartial chess games, has an important property -- the resulting game remains an impartial chess game played on modified partition, where the same moveset remains the same.
\begin{definition}[truncation] \label{def:truncation}
Let $(M,\lambda)$ be a non-terminal impartial chess game. Its \emph{truncation}  $\trunc_M(\lambda)$ is an impartial chess game $(M,\br{\bar \lambda_1,\ldots})$ where $\bar\lambda_i$ is the largest integer such that $(M,\lambda[i-1,\bar \lambda_{i}-1])$ is defined but not terminal.
\end{definition}
\noindent The truncation of a game need not be the smallest partition with the same game tree, as illustrated below. 
\begin{example}
$\trunc_\Knight(\br{5,5,5,3,3,2,1})=(\Knight,\br{4,3,2,1})$ which is partition-equivalent to $(\Knight,\br{3,3,2})$. 
\end{example}

\noindent The following proposition is easily verified. 
\begin{proposition}
    Let $(M,\lambda)$ be a non-terminal impartial chess game. Then $\trunc_M(\lambda[i,j])=\trunc_M(\lambda)[i,j]$.
\end{proposition}

\noindent Using our terminology,
\cite[Theorem 9]{Gottlieb2022LCTR} can be reformulated as follows.
\begin{theorem} \label{thm:truncation}
The $\P_M$-positions of any misère impartial chess game $(M,\lambda)$ exactly correspond to those of $\trunc_M(\lambda)$.
\end{theorem}

\cref{thm:truncation} can also be obtained by using “mis\`ere Grundy numbers” $\mathscr G^-$ as described in Conway~\cite{Con76}. We denote the mis\`ere Grundy number of an impartial chess game $(M, \lambda)$ by $\mathscr G^-_M(\lambda)$. 
The above theorem directly implies that 
$\N_M$-positions also coincide except on the deleted terminal positions of $(M, \lambda)$, as they do not exist in $\trunc_M(\lambda)$.
While the functions $\SG(\trunc(\cdot))$ and $\mathscr{G}^-(\cdot)$ have the same support (with the exception of terminal positions), they may have different values within the support. 

The truncation operation is particularly simple for the game of \textsc{Downright}, \textsc{King}, \textsc{Rook}, and \textsc{Queen}. 
The next corollary follows easily from \cref{def:truncation}.

\begin{corollary}
Let $(M,\lambda)$ be an impartial chess game with $\{(1,0),(0,1)\} \subseteq M$. 
Then $\trunc_M(\lambda)=(M,\lambda^{-})$. 
In particular, this holds for $\Q,\K,\R$ and $\Dr$.
\end{corollary}
The situation for {\sc Knight} and {\sc Pawn} is different because these games do not allow all of the moves of \textsc{Downright}. 
In fact, if $\lambda=\br{\lambda_1,\ldots,\lambda_r}$, then $(\Knight,\lambda)$ is terminal precisely when $\lambda_2\leq 2$ and $\lambda_3\leq 1$. 
Let $\trunc_\Knight(\lambda) = (\Knight, \mu)$. 
If $\lambda_2> 2$ or $\lambda_3>1$, then 
$\mu < \lambda^-$ since more terminal positions are removed. 
Let $\trunc_\Pawn(\lambda) = (\Pawn, \bar \mu)$. Then $(\Pawn,\lambda)$ is terminal precisely when $r=1$. Thus, if $r>1$, then $\bar \mu \leq  \lambda^{-}$ and if $\lambda$ has a repeated entry then the previous inequality is strict.

\subsection{Conway-Gurvich-Ho classification}\label{sec:CG}
Conway~\cite{Con76} introduced the notion of a tame game.  Gurvich and Ho~\cite{GURVICH201854} extended this idea to a more general classification of games. Their definitions, included below for convenience, involve pairs $\C_M(\lambda) = (\SG_M(\lambda), \mathscr G_M^-(\lambda))$, which we refer to as \emph{Conway pairs}. We sometimes omit the subscript when the game is clear from context. If $\C_M(\lambda) \in \{(0, 1), (1, 0)\}$ then $(M, \lambda)$ is called \emph{swap}. If $k \geq 0$ and $\C_M(\lambda) = (k, k)$, then $(M, \lambda)$ is called \emph{symmetric}.

\begin{definition} \label{D.DTP}
A impartial game is called \begin{enumerate} 
\item {\em returnable} if for any move from a $(0,1)$-position (resp.,~a $(1,0)$-position) to a non-terminal position $y$, there is a move from $y$ to a $(0,1)$-position (resp.,~to a $(1,0)$-position).
\item {\em forced} if each move from a $(0,1)$-position results in a $(1,0)$-position and vice versa;
\item {\em domestic} if it has neither $(0,k)$-positions nor $(k,0)$-positions with $k \geq 2$;
\item {\em tame} if it has only $(0,1)$-positions, $(1,0)$-positions, and $(k,k)$-positions with $k \geq 0$;
\item  {\em miserable} if for every position $x$, one of the following holds: (i) $x$ is a $(0, 1)$-position or $(1, 0)$-position, or (ii) there is no move from $x$ to a $(0,1)$-position or $(1, 0)$-position, or (iii) there are moves from $x$ to both a $(0, 1)$-position and a $(1, 0)$-position. 
\item {\em pet} if it has only $(0,1)$-positions, $(1,0)$-positions, and $(k,k)$-positions with $k \geq 2$.
\end{enumerate}
\end{definition}
The relationships between these properties were explored in in \cite{GURVICH201854} and are illustrated in \cref{fig:class}. The \emph{Conway-Gurvich-Ho (CGH) classification} of an impartial game is determined by which of these properties the game has. 

The CGH classification of \textsc{Nim}, \textsc{Subtraction}, \textsc{Wythoff}, \textsc{Mark}, \textsc{LCTR}, and $\Dr$ are known. \textsc{Nim} was shown to be forced and miserable but not pet in \cite{Con76}. \textsc{Subtraction} was shown to be pet and returnable but not forced, and \textsc{Wythoff} was shown to be miserable and returnable, but neither pet nor forced, in \cite{GURVICH201854}. \textsc{LCTR} and $\Dr$ were shown to be domestic but not tame and returnable but not forced in \cite{Gottlieb2022LCTR}. \textsc{Mark} was shown not to be domestic in \cite{GURVICH201854} and is easily shown to be not returnable. Other games are classified in \cite{GURVICH201854}. 

Some of the families of games we study occupy the same region as one of the above games. One family occupies a region not previously known to contain any game; another occupies a region that was only known to contain a single small game. 

In determining the CGH classification for a family of impartial chess games, in most cases we state that the family has one or more properties, and also that it does not have one or more properties. In making the former statement, we assert that the property holds for all members of the family. In the making the latter statement, we assert that there exist family members for which the given property does not hold. 

Let $P\subseteq \Y^+$ and $M$ be an impartial chess moveset. When we state that $(M,P)$ is returnable, we mean that each $(M,\lambda)$ is returnable, where $\lambda \in P$. When we say $(M,P)$ is not returnable, we mean that there exists $\lambda \in P$ such that $(M,\lambda)$ is not returnable.
The same applies to pet, miserable, tame, domestic, and forced. 
Let $G_1$ be a set of impartial games and let $G_2 \subseteq G_1$. If every game in $G_1$ is pet, then so is every game in $G_2$. The same applies to miserable, tame, domestic, forced, and returnable. The following observation is a special case.
\begin{observation}\label{obs:subsetproperty}
Let $P_1\subseteq \Y^+$ and let $ P_2\subseteq P_1$ and let $M$ be an impartial chess moveset. If $\left(M,P_1 \right)$ is pet, then so is $(M,P_2)$. The same applies to miserable, tame, domestic, forced, and returnable. 
\end{observation}

Movesets in which no set of partitions are specified are played on $\Y^+$, e.g., $\Pawn = \left(\Pawn, \Y^+\right)$.
In this paper we will mostly be using sets $P$ which are closed under taking subpartitions\footnote{The only exceptions are the families mentioned in \cref{thm:tameNotMisQueen} and \cref{thm:tameNotMisRook}.}.
In particular, let 
\begin{align*}
    \genstaircase &=\left\{\fgenstaircase rck[i,j] :  r,c,k \in \mathbb Z^{>0}, i,j\in \mathbb Z^{\geq 0}\right\},\\
    \rectangle &=\left\{\frectangle{r}{c} : r,c \in \mathbb Z^{>0}\right\}, \mbox{ and} \\
    \staircase &=\left\{\fstaircase{k} : k \in \mathbb Z^{>0}\right\}.
\end{align*}

Several games are omitted from \cref{fig:class}. For example, by \cref{thm:petreturnable}, $\Pawn$ and $(\Pawn,\staircase)$ are pet and returnable, but not forced. Together with  \cref{obs:subsetproperty} and 
$\staircase\subseteq\genstaircase\subseteq \Y^+$, we see that $(\Pawn, \genstaircase)$ is also pet and returnable, but not forced. For this reason, $(\Pawn,\genstaircase)$ is omitted.
We also omit the results for hooks $\fgama{i}{j}$ and $\lambda$-staircases $S^k_\lambda$, as the former are pet and forced for every moveset studied in this paper, and the latter set is same as $\Y^+$, whose results are listed. 

Throughout the next five subsections, we classify  impartial chess games into respective regions in the same order of appearance as on \cref{fig:class}, where all our results are summarized.

\begin{table}[p]
    \centering
    $\begin{array}{|c|c|c|c|c|}
    \hline
         & \Y^+ & \fgenstaircase{r}{c}{k}  & \frectangle{r}{c} & \fstaircase{k}  \\ \hline
         \Dr & \mathsf R \cap \mathsf D & \mathsf P \cap \mathsf R  & \mathsf P \cap \mathsf F & \mathsf P \cap \mathsf F \\
         \Pawn & \mathsf P\cap \mathsf R & \mathsf P\cap \mathsf R & \mathsf P\cap \mathsf F & \mathsf P \cap \mathsf R\\
        \Knight & \mathsf D\cap \mathsf R & \mathsf D\cap \mathsf R & \mathsf P\cap \mathsf R & \mathsf P\cap \mathsf F  \\
        \R & \mathsf N \cap \mathsf{N'} & \mathsf F\cap \mathsf M & \mathsf F\cap \mathsf M & \mathsf P\cap \mathsf F\\
        \Q & \mathsf N \cap \mathsf{N'} & \mathsf R\cap \mathsf M & \mathsf R\cap \mathsf M & \mathsf P\cap \mathsf F\\
        \Bishop & \mathsf P \cap \mathsf F & \mathsf P \cap \mathsf F & \mathsf P \cap \mathsf F & \mathsf P \cap \mathsf F\\
        \K & \mathsf N \cap \mathsf{N'} & \mathsf N\cap \mathsf{N'} & \mathsf F\cap \mathsf M & \mathsf P\cap \mathsf R \\ \hline
    \end{array}$
    \caption{
    We partition the set of impartial chess games in two ways. The first is into sets $\mathsf{P,M,T,D,N}$  as follows: $\mathsf{P}$ are games which are pet; $\mathsf{M}$ are games which are miserable but not pet; $\mathsf{T}$ are games which are tame but not miserable; $\mathsf{D}$ are games which are domestic but not tame; $\mathsf{N}$ are games which are not domestic. 
    The second is into sets
     $\mathsf{F,R,N'}$, as follows: $\mathsf{F}$ are games which are forced; $\mathsf{R}$ are games which are returnable but not forced; $\mathsf{N'}$ are games which are not returnable.
    }
    \label{tab:class}
\end{table}

\begin{figure}[p]
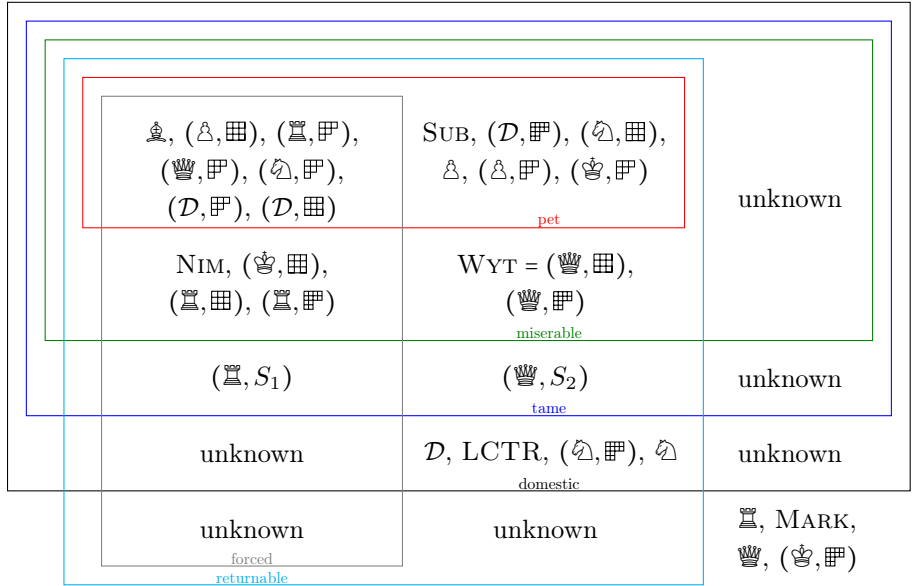

    \centering
    \class
    \caption[fragile]{
    The regions of the CGH classification, together with some
    other already classified games.
    Regions not known by us to be occupied by any game are labeled “unknown”. 
    To our knowledge, our examples of tame but not miserable games are the first infinite families of games of this type, 
    although a single such a game was identified in \cite[Figure 7]{GURVICH201854}.
    We omit games such as 
    $(\Pawn, \genstaircase)$, $(\Bishop, \staircase)$, $(\Bishop, \rectangle)$, $(\Bishop, \genstaircase)$ and $\K$
    as their classifications can be easily derived from \cref{obs:subsetproperty}.
    }
    \label{fig:class}
    
\end{figure}

\subsubsection{Pet games}
Recall that pet games allow only for swap positions and $(k,k)$-positions with $k\ge2$.
\begin{theorem}\label{thm:petforced}
    $\Bishop,(\Pawn,\rectangle ),(\R,\staircase ), (\Queen,\staircase )$, and $(\Knight,\staircase )$
    are pet and forced.
\end{theorem}

\begin{proof} We prove the result for $(\Queen, \staircase)$; the others are similar. 
   By \cref{lem:queenrook} we have $\SG_\Queen\left(\fstaircase{k} \right) =k-1$. There is a similar formula for $\mathscr G^-_\Queen\left( \fstaircase{k} \right)$ from which it follows that
   \[\C_\Queen\left( \fstaircase{k}\right)=\begin{cases}
       (0,1) & \textnormal{if } k=1 \\
       (1,0) & \textnormal{if } k=2 \\
       (k-1,k-1) & \textnormal{if } k>2. \\
   \end{cases}\]
   Since every subpartition of a staircase is a smaller staircase, it follows that $(\Queen,\staircase)$ is pet. The only non-terminal swap position is $\left( \Queen,\fstaircase{2} \right)$ and it can only move to $\left( \Queen,\fstaircase{1} \right)$, so $(\Queen,\staircase)$ is forced.
\end{proof}

\begin{theorem}\label{thm:petreturnable}
$\Pawn,(\Knight,\rectangle )$, $(\K,\staircase)$, $(\Pawn,\staircase)$, and $(\Dr,\genstaircase)$
are pet and returnable but not forced.
\end{theorem}

\begin{proof}
The games  $(\Pawn, \br{2,2,2,1})$, $\left(\Knight,\frectangle 65\right)$, $\left(\K,\fstaircase 4\right)$, $\left(\Pawn,\fstaircase 4\right)$ and$\left(\Dr,\fgenstaircase 122\right)$ show that the respective sets of games to which they belong are not forced. We show that $\left(\Dr, \fgenstaircase rck\right)$ is returnable and pet; the others are similar. 
    
For any subpartition $\lambda$ of $\genstaircase$, we have $\C_\Dr(\lambda)\in A$ where $A=\{(0,1),(1,0),(2,2)\}$. If $(x_1,x_2),(y_1,y_2)\in A$ then $(\Mex(x_1,y_1),\Mex(x_2,y_2))\in A\cup\{(0,0)\}$.
    Thus, it is enough to show that there exists no $(0,0)$-position reachable from $\lambda$. 
    
    Suppose not, and let $\lambda[i,j]$ be a $(0,0)$-position. Then $\lambda[i+1,j]$ and $\lambda[i,j+1]$ are both $(2,2)$-positions. 
    This implies that $\C_\Dr(\lambda[i+2,j])=\C_\Dr(\lambda[i,j+2])$, while $\{\C_\Dr(\lambda[i+2,j]),\C_\Dr(\lambda[i+1,j+1])\}=\{(1,0),(0,1)\}$. We treat the case when $\lambda[i+1,j+1]$ is a $(0,1)$-position. The case when it is a $(1,0)$-position is similar to, and simpler than, the $(0, 1)$-case and is left to the reader.  

Observe that $\lambda[i+1,j+1]$ is not a terminal position.  If it were, then $\lambda= \fgenstaircase 11k=\fstaircase{k}$. Also, $\lambda[i,j+3]$ would be well defined unlike $\lambda[i+1,j+2]$, a contradiction.

        Since $\lambda[i+1,j+1]$ is not a terminal position assume without loss of generality that $\lambda[i+1,j+2]$ is well-defined, which implies, due to $\C_\Dr(\lambda[i,j+2])=(1,0)$, that $\C_\Dr(\lambda[i+1,j+2])=(2,2)$.
        This in turn implies that both $\lambda[i+2,j+2]$ and $\lambda[i+3,j+1]$ are well-defined. Due to $\C_\Dr(\lambda[i+2,j])=(1,0)$, we again obtain $\C_\Dr(\lambda[i+2,j+1])=(2,2)$.
        Then $\C_\Dr(\lambda[i+1,j+1])=(0,1)$ contradicts   
        \begin{align*}
\C_\Dr(\lambda[i+1,j+1])&=\Mex(\C_\Dr(\lambda[i+2,j+1]),\C_\Dr(\lambda[i+1,j+2]))\\
&=\Mex((2,2),(2,2))=(0,0).
        \end{align*}

Since the only Conway pairs possible are $(0,1),(1,0),$ and $(2,2)$, $(\Dr, \genstaircase)$ is pet by definition. Furthermore, there must be a move from any non-terminal $(0,1)$-position to a $(1,0)$-position and vice versa. This implies returnability. 
\end{proof}

\subsubsection{Miserable but not pet games}
Recall that miserable games require any non-swap position to have moves to both or neither of the swap positions.
\begin{theorem} \label{thm:rookrect}
$(\R,\rectangle)$, $(\K,\rectangle)$, and $(\R, \genstaircase)$
    are forced and miserable but not pet.
\end{theorem}

\begin{proof}
    $\C_\R(\frectangle{3}{3}) = \C_\K(\frectangle{3}{3})=\C_\R(\fgenstaircase{3}{3}1) = (0, 0)$, so $(\R, \rectangle)$, $(\K, \rectangle)$, and $(\R, \genstaircase)$ are not pet. 
    We will first show that $(\R, \rectangle)$ is forced and miserable. 

Let $\lambda=\frectangle rc$.
First, we show that the only Conway pairs possible are swap and $(k,k)$ for $k\geq0$. 
To this end, notice that 
\[
(\R, \lambda[r-1,c-1]),(\R, \lambda[r-2,c-1]),(\R, \lambda[r-1,c-2]),\text{ and }(\R, \lambda[r-2,c-2])
\]
are swap positions. In particular, 
\begin{align}
    \C_\R(\lambda[r-1,c-1])=\C_\R(\lambda[r-2,c-2])&=(0,1)\label{eq:swap1}\\
    \C_\R( \lambda[r-1,c-2])=\C_\R( \lambda[r-2,c-1])&=(1,0)\label{eq:swap2}
\end{align} 
whenever the corresponding positions are well-defined. 
Note that any other position can move to both a $(0,1)$ and $(1,0)$ position from \eqref{eq:swap1} and \eqref{eq:swap2}, respectively, or neither.
We now argue that all other positions are symmetric, which implies that $(\R,\rectangle)$ is both forced and miserable.
Suppose not, and let $i<r-2$ and $j<c-2$ be maximal  with respect to $i+j$ such that  
$(\R,\lambda[i,j]))$ is not symmetric. By the choice of $i,j$, the positions reachable from $\lambda[i,j]$ are  all symmetric or  swap. Furthermore, if  either swap position is  reachable, both  are.
It follows that $(\R,\lambda[i,j]))$ is symmetric, a contradiction.
In particular, there are no swap positions other than those described in \eqref{eq:swap1} and \eqref{eq:swap2}.

We can extend the above proof to show that $(\R,\genstaircase)$ is both forced and miserable.
Indeed, we only need to consider the subpositions of $\lambda=\fgenstaircase rck$ which are not of rectangular shape, and show that they are all symmetric.
Again, let $\lambda[i,j]$ be such a subposition, where we chose $i,j$  such that $i+j$ is maximized. By the choice of $i,j$, observe that all reachable positions consist of symmetric ones, and in addition they may they include both swap positions $(0,1)$ and $(1,0)$. By the same reasoning it follows that $(\R,\lambda[i,j]))$ is symmetric, a contradiction.
 \end{proof}

As noted in \cref{sec:queen}, $(\Queen,\rectangle )$ is equivalent to the well-known game \textsc{Wythoff}. It was shown in \cite[Propositions 10 and 11]{GURVICH201854} that \textsc{Wythoff} is miserable but not pet and returnable but not forced. It is more involved to show that $(\Queen,\genstaircase )$ is in the same region; the cases when $r=c=2$, when $r=2$ and $c=3$, and when $r\geq 3$ and $c \geq 3$ require separate treatment. We summarize these facts in the following theorem, the proof of which we omit in the interest of brevity. 

\begin{theorem}
    $(\Queen,\genstaircase )$ and $(\Queen,\rectangle )$ 
    are miserable but not pet and returnable but not forced.
\end{theorem}
\subsubsection{Tame but not miserable games}
Recall that tame games allow only for swap and symmetric positions. Gurvich and Ho found a small game that is forced and tame but not miserable; see Figure 7 of \cite{GURVICH201854}.
Let  $S_1 =\{\br{\ell + 1, \ell^k}[i, j] : \ell \geq 3, k > \ell, i,j \in \mathbb Z^{\geq0}\}$. We show next that the games $(\R, S_1)$ belong to the same region.

\begin{theorem}\label{thm:tameNotMisRook}
$(\R, S_1)$ is forced and tame but not miserable. 
\end{theorem}

\begin{proof}
By \cref{thm:rookrect}, $(\R,\rectangle)$ is forced and miserable. In particular, it contains only swap and $(t,t)$-positions. 
Fix any $\br{\ell+1,\ell^k}\in S_1$ with $k > \ell \geq 3$. The subposition $(\R, \br{\ell + 1, \ell^k})[k-\ell,\ell]=(\R, \br{\ell^{\ell+1}})=(\R, \frectangle{\ell+1}{\ell})$ is well-defined in all cases. 
This subposition contains both swap positions in each of the last two columns. All other columns contain a $(0,0)$- and a $(1,1)$-position but neither swap position. Also, $(\R, \br{\ell + 1, \ell^k})[0,\ell]$ is a $(0,1)$-position. Thus $(\R, \br{\ell + 1, \ell^k})[i,j])$ is a $(t,t)$-position for some $t\geq2$ and for all  $i\in [0,k-l]$ and $j\in [0 ,\ell]$. 

This implies $(\R, S_1)$ is tame. However, since $\ell\geq 3$, there is a move from the starting position, to a $(0,1)$-position, but not to a $(1,0)$-position, hence it is not miserable.
The only nonterminal swap positions are in the subposition $(\R, \br{\ell + 1, \ell^k})[k-1,\ell-2])=(\R,\br{2^2})$. This subposition is forced.
\end{proof}

The region that is tame but not miserable and returnable but not forced was not known by us to contain any game.
We show that it is occupied by $(\Queen, S_2)$, where $S_2=\{\br{5, 4^{k}}[i,j] : k \geq 7, i,j \in \mathbb Z^{\geq0}\}$.

\begin{theorem}\label{thm:tameNotMisQueen}
$(\Queen, S_2)$ is tame but not miserable and returnable but not forced. 
\end{theorem}
\begin{proof}
Any fixed $\br{5,4^k}\in S_2$ with $k \geq 7$ contains the subpartition $\br{5, 4^{k}}[k-6,0]=\br{4^7}$.
Also, $(\Queen, \br{4^7})$ contains only swap and $(t,t)$-positions, where the first column contains both $(0,0)$- and $(1,1)$-positions. The first column does not admit any swap position while the other columns contain both swap positions.
Further, $(\Queen,\br{5, 4^{k}}[0,4])$ is a $(0,1)$-position. 
Thus $(\Queen,\br{5, 4^{k}}[i,j])$ is a $(t,t)$-position for some $t\geq2$ and for all $i\in [0,k-7]$ and $j\in [0,3]$.

This implies $(\Queen, S_2)$ is tame. However, there is a move from the starting position to a $(0,1)$-position, but not to a $(1,0)$-position, hence it is not miserable.
The only nonterminal swap positions are in the subposition $(\Queen,\br{5, 4^{k}}[k-4,1])=(\Queen,\br{3^3})$. This subposition is returnable but not forced.
\end{proof}

\subsubsection{Domestic but not tame games}
Recall that in a domestic game, a value $0$ can appear in a Conway pair only if the position is either swap or symmetric.
\begin{theorem}
$(\Knight,\genstaircase )$ and $\Knight$
    are domestic but not tame and returnable but not forced. 
\end{theorem}

\begin{proof}
    To see that $\Knight$ and $(\Knight, \genstaircase)$ are not tame, note that $\C_\Knight \left( \fgenstaircase 66{2}[1,0] \right) = (2,1)$. \cref{thm:petreturnable} shows that $\Knight$ and $(\Knight, \genstaircase)$ are not forced. 
    
    To see that $\Knight$ and $(\Knight, \genstaircase)$ are domestic, note that there are at most two moves from any position, so neither coordinate of the Conway pair can be greater than two. Thus it suffices to show that $(2,0)$ and $(0,2)$ cannot occur. 

    The proof is by contradiction. Let $\lambda$ be a minimal partition (with respect to the Young lattice) so that $\C_\Knight(\lambda) = (0,2)$ or $\C_\Knight(\lambda) = (2,0)$. We treat the case $\C_\Knight(\lambda)=(0,2)$; the case $\C_\Knight(\lambda)=(2,0)$ is similar and is left to the reader. 

    If $\C_\Knight(\lambda)=(0,2)$ then, without loss of generality, it must be true that $\C_\Knight(\lambda[2,1])=(x,1)$ and $\C_\Knight(\lambda[1,2])=(y,0)$ for some $x, y \in \{1, 2\}$. Due to minimality, $y$ cannot be 2, so $y = 1$. 

    If $x = 2$, then $\C_\Knight(\lambda[3,3])=(0,2)$, contradicting minimality, so $x = 1$. 
    
    If $\lambda[3,3]$ is not well-defined, then $\lambda[4,2]$ must be well-defined since $\lambda[2,1]$ cannot be terminal, so $\lambda[4,2]$ must be terminal and $\C_\Knight(\lambda[4,2])=(0,1)$, contradicting $\C_\Knight(\lambda[2,1])=(1,1)$.
    Thus $\lambda[3,3]$ is well-defined and
    $\C_\Knight(\lambda[3,3])=(z,2)$ for some $z \in \{0, 2\}$. Due to minimality $z$ cannot be $0$, so $z = 2$. It follows that $\C_\Knight(\lambda[4,2])=(0,0)$, which implies 
    $\C_\Knight(\lambda[5,4])=(1,1)$, from which it follows that $\C_\Knight(\lambda[4,5])=(0,0)$. This implies 
    $\C_\Knight(\lambda[2,4])=(1,1)$, a contradiction with $\C_\Knight(\lambda[1,2])=(1,0)$. This concludes the proof for domesticity. 

    To see that $\Knight$ is returnable (which implies that $(\Knight, \genstaircase)$ is returnable), consider a non-terminal partition $\lambda$ with $\C_\Knight(\lambda) = (0,1)$. Since there are no $(k,0)$-positions with $k\geq2$, it must be possible to move to a $(1,0)$-position. Similarly a $(1,0)$ position can always move to a $(0,1)$-position. Without loss of generality, suppose that $\C_\Knight(\lambda[1,2])=(1, 0)$. Then $\C_\Knight(\lambda[2,1]) \in \{(1,2), (1, 0), (2,2)\}$. If $\C_\Knight(\lambda[2,1])= (1,0)$ then $\Knight(\lambda)$ is returnable. Also, if $\C_\Knight(\lambda[2,1])=(2,2)$ then $\C_\Knight(\lambda[3,3])=(0,1)$, so $\Knight(\lambda)$ is again returnable. 

 Finally, suppose $\C_\Knight(\lambda[2,1])=(1,2)$. If $\C_\Knight(\lambda[3,3]) = (0,1)$, then $\Knight(\lambda)$ is returnable. The only alternative is that $\C_\Knight(\lambda[3,3]) = (2,1)$. In this case, we must have $\C_\Knight(\lambda[4,2]) = (0,0)$. It follows that $\C_\Knight(\lambda[5,4]) = (1,2)$. From this, we obtain $\C_\Knight(\lambda[4,5) = (0,0)$.
 Since non-terminal swap positions must move to at least another swap it follows that $\C_\Knight(\lambda[2,4]) = (0, 1)$ which is a contradiction with  $\C_\Knight(\lambda[4,5]) = (0,0)$. 
\end{proof}

\subsubsection{Non-domestic games}

\begin{theorem}
    $\R,\Queen$, and $(\K,\genstaircase)$
    are neither domestic nor returnable.
\end{theorem}

\begin{proof} $\R$, $\Queen$, and $(\K, \genstaircase)$  are not domestic because $\C_\R(\br{5,4^3})=(4, 0)$ and $\C_\Queen(\br{9,8^4})=(0,10)$ and $\C_\K(\fgenstaircase 224[0,1])=(3,0)$, respectively. 

$\R$ is not returnable since $\C_\R(\br{6, 5^4})= (0,1)$ and there is no move from $(\R,\br{6, 5^4}[1,0])$ to a $(0,1)$-position. $\Queen$ is not returnable since $\C_\Queen(\br{11^{13}, 10^4})= (0,1)$ and there is no move from $(\Queen,\br{11^{13}, 10^4}[12,0])$ to a $(0,1)$-position. $(\K, \genstaircase)$ is not returnable since $\C_\K(\fgenstaircase 432[1,0]) = (1,0)$ and there is no move from $(\K,\fgenstaircase 432[2,0])$ to a $(1,0)$-position. 
\end{proof}

\section{Open questions in impartial chess} \label{sec:questions}

Hooks (see \cref{def:partition-families})
are a parametrized class of partitions studied in \cite{basic2022some,Basic_2023,Ilic2019}.
The restriction of impartial chess games to hooks is not very interesting from the point of view of CGH classification since they are forced and pet for every chess piece.

A more interesting family to consider are \emph{thick hooks}  \cite{basic2022some,Basic_2023} of the form $\br{c^a, b^{r-a}}$,  which are simultaneously subpartitions of generalized staircases and generalizations of both rectangles and hooks. Computing the Sprague-Grundy values of generalized staircases might be out of reach for \textsc{King} and \textsc{Rook}, but we believe it could be possible for thick hooks.

For rectangles, the coefficient 2 in \cref{lem:double} can be replaced with the golden ratio \cite{Wythoff}. It would be interesting to know if the coefficient 2 can be reduced for arbitrary partitions. Showing that it cannot be reduced would require finding an infinite family of partitions such that the inequality in \cref{lem:double} holds. Two such partitions are $\br{3^2}$ and $\br{9^2,7^3}$.

\begin{conjecture}
Let $r$ be odd and $k \geq 1$. Then
\[\SG_\K(\fgenstaircase  rrk)=\begin{cases}
0 & \text{$k \bmod (r+2)$ is odd} \\ 
1 & \text{$k \bmod (r+2)$ is even and not 0} \\
2 & \text{$k \bmod (r+2)$ is 0}.
\end{cases}\]
\end{conjecture}

\begin{conjecture}
Let $r\neq c$ and both $r,c$ be even and $k \geq 1$. 
Then
\[\SG_\K(\fgenstaircase  rck)=\begin{cases}
2 & \text{ if  $k$ is odd} \\
3 & \text{ if $k$ is even}.
\end{cases}\]
\end{conjecture}

The Sprague-Grundy values of $(\R,\rectangle)$ and $(\R, \genstaircase)$ are given to us by two-pile \textsc{Nim} and \cref{thm:rookgenstairsg}, respectively. The $\SG$-values of \textsc{Rook} on general partitions are unknown.
$\P/\N$-positions for \textsc{Queen} remain open. Given the complicated resolution of these positions on rectangles, significant progress on more complicated partitions seems unlikely. 

\paragraph{Acknowledgements}
This work is inspired by the recent legacy of the late Elwyn Berlekamp and his pioneering work on impartial chess.

This work was supported in part by the Slovenian Research and Innovation Agency (research projects J1-4008, J1-3002, research program P1-0383, and bilateral project BI-US-24-26-018).
The first author was supported by an internal grant from Rhodes College. 

\pagebreak
\bibliographystyle{elsarticle-num} 

\end{document}